\documentclass[11pt,a4paper]{article}

\usepackage[a4paper,margin=1in]{geometry}
\usepackage{amsmath,amsthm,amssymb}
\usepackage{graphicx,color}
\usepackage{enumerate}
\usepackage{xspace}
\usepackage{nicefrac}
\usepackage{url}
\usepackage[pagebackref]{hyperref}
\definecolor{linkblue}{rgb}{0.1,0.1,0.8}
\hypersetup{
pdftitle={Expanders Are Universal for the Class of All Spanning Trees},
pdfauthor={Daniel Johannsen, Michael Krivelevich, and Wojciech Samotij},
colorlinks=true,
linkcolor=linkblue,
filecolor=linkblue,
urlcolor=linkblue,
citecolor=linkblue,
pdfstartview={FitH},
}
\renewcommand*{\backref}[1]{}
\renewcommand*{\backrefalt}[4]{
\ifcase #1 (Not cited.)
\or (Cited on page~#2.)
\else  (Cited on pages~#2.)
\fi
}

\newcommand{\ignore}[1]{}

\newtheorem{theorem}{Theorem}[section]
\newtheorem{lemma}[theorem]{Lemma}
\newtheorem{corollary}[theorem]{Corollary}
\newtheorem{proposition}[theorem]{Proposition}

\theoremstyle{definition}
\newtheorem{definition}[theorem]{Definition}

\newcommand{\secref}[1]{Section~\ref{#1}\xspace}
\newcommand{\thmref}[1]{Theorem~\ref{#1}\xspace}
\newcommand{\lemref}[1]{Lemma~\ref{#1}\xspace}
\newcommand{\propref}[1]{Proposition~\ref{#1}\xspace}
\newcommand{\corref}[1]{Corollary~\ref{#1}\xspace}
\newcommand{\defref}[1]{Definition~\ref{#1}\xspace}

\newcommand{\ie}{that is\xspace}
\newcommand{\EXP}{\mathbb{E}}
\newcommand{\Bin}{\textsc{Bin}}
\newcommand{\Hyp}{\textsc{Hyp}}
\newcommand{\ea}{\ensuremath{\textrm{E1}}}
\newcommand{\eb}{\ensuremath{\textrm{E2}}}
\newcommand{\nea}{\ensuremath{\overline{\textrm{E1}}}}
\newcommand{\neb}{\ensuremath{\overline{\textrm{E2}}}}

\newcommand{\eps}{\varepsilon}
\newcommand{\euler}{\mathrm{e}}
\renewcommand{\phi}{\varphi}

\title{Expanders Are Universal for the Class of All Spanning Trees}
\author{
  Daniel Johannsen$^{(1),}$\thanks{Daniel Johannsen is supported by a fellowship within the Postdoc-Programme of the German Academic Exchange Service (DAAD).}
  \and
  Michael Krivelevich$^{(1),}$\thanks{Research supported in part by a USA-Israel BSF grant and by a grant from the Israel Science Foundation.}
  \and
  Wojciech Samotij$^{(1),(2),}$\thanks{The research of Wojciech Samotij is supported by the ERC Advanced Grant DMMCA.}
  \and \\
  $^{(1)}$: School of Mathematical Sciences\\
  Tel Aviv University \\
  Tel Aviv, 69978\\
  Israel
  \and \\
  $^{(2)}$: Trinity College\\
  Cambridge CB2 1TQ\\
  United Kingdom
}
\date{July 23, 2011}

\begin{document}

\maketitle

\begin{abstract}
Given a class of graphs $\mathcal{F}$, we say that a graph $G$ is \emph{universal for $\mathcal{F}$}, or \emph{$\mathcal{F}$-universal}, if every $H \in \mathcal{F}$ is contained in $G$ as a subgraph. The construction of sparse universal graphs for various families $\mathcal{F}$ has received a considerable amount of attention. One is particularly interested in tight $\mathcal{F}$-universal graphs, i.\,e., graphs whose number of vertices is equal to the largest number of vertices in a graph from $\mathcal{F}$. Arguably, the most studied case is that when $\mathcal{F}$ is some class of trees.

Given integers $n$ and $\Delta$, we denote by $\mathcal{T}(n,\Delta)$ the class of all $n$-vertex trees with maximum degree at most $\Delta$. In this work, we show that every $n$-vertex graph satisfying certain natural expansion properties is~$\mathcal{T}(n,\Delta)$-universal or, in other words, contains every spanning tree of maximum degree at most $\Delta$. Our methods also apply to the case when~$\Delta$ is some function of $n$. The result has a few very interesting implications. Most importantly, since random graphs are known to be good expanders, we obtain that the random graph $G(n,p)$ is asymptotically almost surely (a.a.s.) universal for the class of all bounded degree spanning (i.\,e., $n$-vertex) trees provided that $p\ge c n^{-1/3}\log^2n$ where~$c>0$ is a constant. Moreover, a corresponding result holds for the random regular graph of degree $p n$. In fact, we show that if $\Delta$ satisfies $\log n \leq \Delta \leq n^{1/3}$, then the random graph $G(n,p)$ with $p \ge c\Delta n^{-1/3}\log n$ and the random $r$-regular $n$-vertex graph with $r\ge  c\Delta n^{2/3}\log n$ are a.a.s.~$\mathcal{T}(n,\Delta)$-universal. Another interesting consequence is the existence of locally sparse $n$-vertex $\mathcal{T}(n,\Delta)$-universal graphs. For~$\Delta\in O(1)$, we show that one can (randomly) construct $n$-vertex $\mathcal{T}(n,\Delta)$-universal graphs with clique number at most five. This complements the construction of Bhatt, Chung, Leighton, and Rosenberg (1989), whose $\mathcal{T}(n,\Delta)$-universal graphs with merely $O(n)$ edges contain large cliques of size~$\Omega(\Delta)$. Finally, we show robustness of random graphs with respect to being universal for $\mathcal{T}(n,\Delta)$ in the context of the Maker-Breaker tree-universality game.
\end{abstract}

\newpage
\section{Introduction}
\label{sec:introduction}
A graph~$G$ is \emph{universal} for a class of graphs~$\mathcal{F}$ (equivalently, we say that $G$ is \emph{$\mathcal{F}$-universal}) if a copy of every member of~$\mathcal{F}$ is contained in~$G$. Since $\mathcal{F}$-universality implies that the maximum degree of~$G$ is at least as large as the maximum degrees of all graphs in~$\mathcal{F}$, it is natural to consider only classes with bounded maximum degree. There exists a rich literature on explicit and randomized constructions of universal graphs~\cite{AlAs02,ChGrPi78,BaChErGrSp82,BhChLeRo89,ChGr78,ChGr79,ChGr83,FrPi87,Ro81,Ha01,AlCa07,Ca10,CaKo99,AlCaKoRoRu00}. One of the classes for which universality has been studied extensively is the class of bounded degree trees. For a positive integer~$n$ and a positive real number~$\Delta$, let~$\mathcal{T}(n,\Delta)$ be the class of all $n$-vertex trees with maximum degree at most~$\Delta$. Bhatt, Chung, Leighton, and Rosenberg~\cite{BhChLeRo89} gave an explicit construction of very sparse $n$-vertex $\mathcal{T}(n,\Delta)$-universal graphs of maximum degree bounded by a function in~$\Delta$. For~$\Delta\in O(1)$, their universal graphs have only~$O(n)$ edges.

In this work, instead of constructing specific universal graphs, we are rather interested in determining for which edge densities almost all $n$-vertex graphs become $\mathcal{T}(n,\Delta)$-universal. In particular, we want to know for which edge probabilities~$p$, the binomial random graph~$G(n,p)$ asymptotically almost surely (a.a.s.) becomes universal for the class of all bounded degree spanning trees. Moreover, we want to identify particular pseudo-random properties that guarantee this universality and that are a.a.s.\ satisfied by~$G(n,p)$.

Since every connected component of a graph contains a spanning tree, we know that, for large values of~$c$, the random graph $G(n,c/n)$ a.a.s.\ contains a copy of some tree that covers a significant proportion of the vertices of~$G$. We may now ask whether this is true for every \emph{specific} tree with size linear in~$n$. For paths, \ie, trees with maximum degree two, Ajtai, Koml{\'o}s and Szemer{\'e}di~\cite{AjKoSz81} showed that if~$c>1$ the random graph $G(n,c/n)$ indeed contains a.a.s.\ a path of length linear in~$n$. On the other hand, for fixed~$c$, the random graph $G(n,c/n)$ a.a.s.\ has maximum degree $(1+o(1))\log n/\log\log n$ and therefore we cannot expect to embed trees of larger maximum degree. Thus, a more reasonable question is to ask whether for every bounded degree tree~$T$ with size linear in~$n$, the random graph $G(n,c/n)$ a.a.s.\ contains a copy of~$T$. 

This question was first addressed by Fernandez de la Vega~\cite{Fe88}, who showed that, for fixed~$\Delta\ge 2$ and~$7/8\le\alpha<1$, there exists a constant $c=c(\Delta,\alpha)$ with~$\Delta-1<c\le 8(\Delta-1)$ such that, for every specific tree~$T\in\mathcal{T}((1-\alpha)n,\Delta)$, the random graph~$G(n,c/n)$ a.a.s.\ contains a copy of~$T$. Alon, Krivelevich, and Sudakov~\cite{AlKrSu07} showed for all~$\eps\in(0,1)$ the existence of a constant $c=c(\Delta,\eps)$ such that $G(n,c/n)$ a.a.s.\ contains a copy of all trees in~$\mathcal{T}((1-\eps)n,\Delta)$, thus extending the result Fernandez de la Vega to \emph{almost spanning trees}, \ie, to arbitrary small values of~$\eps$. A better bound for $c(\Delta,\eps)$ and a resilience version of this result were obtained by Balogh, Csaba, Pei, and Samotij~\cite{BaCsPeSa10} and by Balogh, Csaba, and Samotij~\cite{BaCsSa11}, respectively.

Besides being valid for small values of~$\eps$, the results of Alon et~al.\ as well as of Balogh et~al.\ exhibit a substantial difference to that of Fernandez de la Vega. Instead of~(i) showing that a \emph{fixed} tree in~$\mathcal{T}((1-\eps)n,\Delta)$ is a.a.s.\ contained in $G(n,p)$, they show instead that~(ii) $G(n,p)$ is a.a.s.\ \emph{universal} for the whole class~$\mathcal{T}((1-\eps)n,\Delta)$, \ie, contains a copy of every tree in~$\mathcal{T}((1-\eps)n,\Delta)$ simultaneously.  Note that,  for~$\Delta\ge 3$, the size of~$\mathcal{T}((1-\eps)n,\Delta)$ is exponential in~$n$ and therefore the union-bound is not sufficient to derive~(ii) from~(i).

In order to show that $G(n,c/n)$ is a.a.s.\ universal for $\mathcal{T}((1-\varepsilon)n,\Delta)$, both Alon et al.~\cite{AlKrSu07} and Balogh, Csaba, Pei, and Samotij~\cite{BaCsPeSa10} showed that $G(n,c/n)$ exhibits certain pseudo-random properties that imply large expansion of small sets of vertices (after one deletes few vertices with very small degrees). This allows to apply the classical tree-embedding result of Friedman and Pippenger~\cite{FrPi87} (as was done in~\cite{AlKrSu07}) or its somewhat stronger version due to Haxell~\cite{Ha01} (as was done in~\cite{BaCsPeSa10}) to embed every bounded degree tree that covers all but an $\varepsilon$-fraction of the vertices of $G(n,c/n)$. Recently, Sudakov and Vondr{\'a}k~\cite{SuVo10} gave a randomized algorithm to efficiently embed bounded degree almost spanning trees in graphs with certain expansion properties. We discuss the result of Haxell in \secref{sec:expander}.

For~$p n>\log n$, the situation changes drastically. In this regime, the random graph $G(n,p)$ is connected and we may ask for the existence of \emph{spanning} trees. The very specific case of embedding a Hamilton path was resolved by Koml{\'o}s and Szemer{\'e}di~\cite{KoSz83} and, independently, Bollob{\'a}s~\cite{Bo84}, who proved that if $p n \geq\log n + \log\log n + \omega(1)$ (where $\omega(1)$ is any function which tends to infinity as $n\to\infty$), then $G(n,p)$ a.a.s.\ contains a Hamilton cycle. Frieze and Krivelevich~\cite{FrKr02} and Krivelevich and Sudakov~\cite{KrSu02} investigated pseudo-random conditions expressed in terms of the spectral gap of the host graph which guarantee the existence of Hamilton paths.  Hefetz, Krivelevich, and Szab{\'o}~\cite{HeKrSz09} showed \emph{Hamilton-connectedness} (\ie, the existence of a Hamilton path between any two vertices) of graphs with expansion properties similar to those we introduce in the next section. We discuss their result in \secref{sec:expander}.

Addressing the question of embedding a \emph{fixed} tree~$T\in\mathcal{T}(n,\Delta)$, Krivelevich~\cite{Kr10} showed that if $np\ge\frac{40}{\eps}\Delta\log n+n^{\eps}$ for some~$\eps>0$, then the random graph~$G(n,p)$ a.a.s.\ contains a copy of~$T$. Moreover, it is shown in~\cite{Kr10} that this bound on $p$ is asymptotically tight in the order of magnitude if $n^\eps \leq \Delta\leq n/\log n$. Extending and improving a result of Alon, Krivelevich, and Sudakov~\cite{AlKrSu07}, Hefetz, Krivelevich, and Szab{\'o}~\cite{HeKrSz11} showed that if, in addition,~$T$ has a linear number of leaves or contains a bare path (\ie, a path in which all vertices have degree two in~$T$) of length linear in~$n$, then~$G(n,p)$ a.a.s.\ contains a copy of~$T$ already for $p n=(1+o(1))\log n$.

To the best of our knowledge, until now there exist no results directly addressing the question of $\mathcal{T}(n,\Delta)$-universality of~$G(n,p)$ for~$p\in o(1)$. For~$\mathcal{G}(n,\Delta)$, the class of all graphs on~$n$ vertices with maximum degree at most~$\Delta$, Dellamonica, Kohayakawa, R\"{o}dl, and Ruci{\'{n}}ski~\cite{DeKoRoRu08} showed that there exists a constant~$c:=c(\Delta)$ such that, for $p n\ge c n^{1-1/(2\Delta)}\log^{1/\Delta}n$, the random graph $G(n,p)$ is a.a.s.\ $\mathcal{G}(n,\Delta)$-universal, improving an earlier result in~\cite{AlCaKoRoRu00}. As a special case, this result also applies to the subclass $\mathcal{T}(n,\Delta)$ and implies that the random graph $G(n,p)$ is a.a.s.\ $\mathcal{T}(n,\Delta)$-universal for such values of~$p$. Recently, Dellamonica et~al.~\cite{DeKoRoRu11} improved their result to $p n\ge c n^{1-1/\Delta}\log n$. However, in all these bounds~$p n=n^{1-o(1)}$ for~$\Delta\to\infty$.

\subsection{Outline}
This paper is organized as follows: In the next section, we present all results of this work and put them into context with existing research. The subsequent sections are devoted to the proofs of these results. We give respective references to these proofs whenever we state a result in the next section.

\subsection{Notation}
Let $\mathbb{N}$ and~$\mathbb{R}^{+}$ be the sets of positive integers and positive real numbers, respectively. For two functions $f,g\colon\mathbb{N}\to\mathbb{R}^+$, we write $f\in o(g)$ or, equivalently, $f\ll g$, to denote the fact that $\lim_{n\to\infty}\frac{f(n)}{g(n)}=0$ and $f\in O(g)$ or, equivalently, $g\in\Omega(f)$, to denote the fact that there exists an absolute constant~$c\in\mathbb{R}^+$ such that $f(n)\le c g(n)$ for all~$n\in\mathbb{N}$.

Given a graph~$G$, we denote its vertex set by~$V(G)$ and its edge set by~$E(G)$. All graphs considered in this work are finite (where typically~$|V(G)|$ will be denoted by~$n$), simple, and undirected. For~$X\subseteq V(G)$, let~$G[X]$ be the subgraph of~$G$ induced by~$X$ and let $N_G(X)$ be the \emph{external} neighborhood of~$X$, \ie, $N_G(X)=|\{y\in V(G)\setminus X\mid \exists\,x\in X \colon \{x,y\}\in E(G)\}|$. For~$v\in V(G)$, let $N_G(v):=N_G(\{v\})$ and let $\deg_G(v):=|N_G(v)|$. Given two sets~$X,Y\subseteq V(G)$, $e_G(X,Y)$ is the number of ordered pairs $(x,y)$ with~$x\in X$ and~$y\in Y$ such that $\{x,y\}\in E(G)$. Note that if~$X$ and~$Y$ intersect, then all edges in the intersection are counted twice.

Finally, for two graphs~$H$ and~$G$, an \emph{embedding}~$\phi$ of~$H$ in~$G$ is an injective graph homomorphism, \ie, an injective map~$\phi\colon V(H)\to V(G)$ such that~$\{v,w\}\in E(H)$ implies~$\{\phi(v),\phi(w)\}\in E(G)$. We say~$\phi$ embeds~$H$ \emph{onto}~$G$ if~$\phi$ is bijective and we say~$G$ \emph{contains a copy of}~$H$ if there exists an embedding of~$H$ in~$G$.

\section{Our Results}
\label{sec:results}

\subsection{Tree-Universality of \texorpdfstring{$(n,d)$}{(n,d)}-Expanders}
The main contribution of this work is establishing tree-universality for the members of a certain class of graphs with good expansion properties, which we term \emph{$(n,d)$-expanders}. For all~$n\in\mathbb{N}$ and all~$d\in\mathbb{R}^+$, let
\[
m(n,d):=\Big\lceil\frac{n}{2d}\Big\rceil.
\]
The following notion is an adaptation of the expansion properties investigated in~\cite{BaCsPeSa10} and~\cite{HeKrSz09}.
\begin{definition}[$(n,d)$-expander]
\label{def:expander}
Let~$n\in\mathbb{N}$, let~$d\in\mathbb{R}^+$, and let~$m:=m(n,d)$. A graph $G$ is an $(n,d)$\emph{-expander} if~$|V(G)|=n$ and~$G$ satisfies the following two conditions:
\begin{itemize}
\item[(\ea):] $|N_G(X)|\ge d|X|$ for all $X\subseteq V(G)$ with $1\le |X|< m$.
\item[(\eb):] $e_G(X,Y)>0$ for all disjoint $X,Y\subseteq V(G)$ with $|X|=|Y|=m$.
\end{itemize}
\end{definition}
A simple calculation (\lemref{lem:monotone} in \secref{sec:expander}) shows that these properties are monotone, \ie, if $3\le d_0\le d\le n/6$, then every $(n,d)$-expander is also an $(n,d_0)$-expander.

The main result of this work is the following theorem which states that $(n,d)$-expanders are tree-universal.
\begin{theorem}[Tree-Universality]
\label{thm:universality}
There exists an absolute constant~$c\in\mathbb{R}^+$ such that the following statement holds. Let~$n\in\mathbb{N}$ and~$\Delta\in\mathbb{R}^+$ satisfy $\log n\le\Delta\le c n^{1/3}$. Then every $(n,7\Delta n^{2/3})$-expander is universal for $\mathcal{T}(n,\Delta)$.
\end{theorem}
Note that the class $\mathcal{T}(n,\log n)$ also includes all $n$-vertex trees of maximum degree~$\Delta$ smaller than $\log n$ (for example, for~$\Delta\in O(1)$ and~$n$ large). Thus, \thmref{thm:universality} also applies to the situation where~$\Delta<\log n$ by setting~$\Delta$ to~$\log n$. The proof of \thmref{thm:universality} is given in \secref{sec:universality}.

There are many known constructions of expanders (see, e.g.~\cite{KrSu06}). Thus, by verifying the conditions in \defref{def:expander}, one obtains explicit constructions of relatively sparse universal graphs for~$\mathcal{T}(n,\Delta)$.

\subsection{Random Graphs}
Random graphs are well-known to typically exhibit strong expansion properties. For example, the random graph~$G(n,p)$ with $p n\ge 7d\log n$ is a.a.s.\ an $(n,d)$-expander (\lemref{lem:gnp} in \secref{sec:random}) and thus, as a direct consequence of \thmref{thm:universality}, tree-universal.
\begin{theorem}
\label{thm:gnp}
There exists an absolute constant~$c\in\mathbb{R}^+$ such that the following statement holds. Let~$\Delta\colon\mathbb{N}\to\mathbb{R}^+$ satisfy $\Delta\ge\log n$. Then the random graph $G(n,p)$ is a.a.s.\ universal for $\mathcal{T}(n,\Delta)$, provided that $p n\ge c\Delta n^{2/3}\log n$.
\end{theorem}
\thmref{thm:gnp} implies that $G(n,c n^{-1/3}\log^2 n)$ is a.a.s.\ universal for $\mathcal{T}(n,\log n)$ and thus also for $\mathcal{T}(n,O(1))$ if~$c$ is a large enough constant.

Likewise, the random $r$-regular $n$-vertex graph with $\max\{7d\log n,\sqrt{n}\log n\}\le r\ll n$  (where $rn$ is even) is a.a.s.\ an $(n,d)$-expander (\lemref{lem:regular} in \secref{sec:random}) and thus, as a direct consequence of \thmref{thm:universality}, tree-universal.
\begin{theorem}
\label{thm:regular}
There exists an absolute constant~$c\in\mathbb{R}^+$ such that the following statement holds. Let~$\Delta\colon\mathbb{N}\to\mathbb{R}^+$ satisfy $\Delta\ge\log n$. Then the $r$-regular random graph on $n$~vertices is a.a.s.\ universal for $\mathcal{T}(n,\Delta)$, provided that $c\Delta n^{2/3}\log n\le r\ll n$, and $r n$ is even.
\end{theorem}
 Note that the restriction~$r\ll n$ in \thmref{thm:regular} is likely to be an artifact resulting from the use of a switching argument in the proof of \lemref{lem:regular} and we believe the result should extend to linear values of $r=r(n)$. Similarly as before, \thmref{thm:regular} implies that the random $r$-regular graph on~$n$ vertices with $r\ge c n^{2/3}\log^2 $ (where $rn$ is even) is a.a.s.\ universal for $\mathcal{T}(n,O(1))$ if~$c$ is a large enough constant.

\subsection{Locally Sparse Expanders}
Bhatt, Chung, Leighton, and Rosenberg~\cite{BhChLeRo89} gave an explicit construction of $\mathcal{T}(n,\Delta)$-universal graphs on $n$ vertices whose maximum degree is bounded by a function of~$\Delta$. Thus, for constant~$\Delta$, the number of edges in this graph is in~$O(n)$. In comparison, the random graph we consider in \thmref{thm:gnp} a.a.s.\ has~$\Theta(n^{5/3}\log^2 n)$ edges. However, the graph constructed in~\cite{BhChLeRo89} is \emph{locally dense}, \ie, it contains a large number of cliques of size~$\Omega(\Delta)$ (cf. Lemma~8 in~\cite{BhChLeRo89}). In this section, we show how to construct \emph{locally sparse} graphs that are universal for all bounded degree trees.

Consider the random graph~$G$ drawn according to~$G(n,p)$. The expected number of cliques of size~$k$ in~$G$ is~$\binom{n}{k}p^{\binom{k}{2}}$. Therefore, by Markov's inequality, if~$n^kp^{\frac{k(k-1)}{2}}\ll 1$, then~$G$ a.a.s.\ does not contain any clique of size~$k$. Consequently, for $p(n)= n^{-1/3}\log^2 n$, the random graph~$G(n,p)$ is a.a.s.\ both $K_8$-free and~$\mathcal{T}(n,\log n)$-universal. Thus, for sufficiently large~$n$, a $\mathcal{T}(n,\log n)$-universal graph with clique number at most seven exists. We can strengthen this observation by showing that, for an appropriate choice of chosen~$d$, $r$, and~$p$, the random graph~$G(n,p)$ is still a.a.s. an $(n,d)$-expander even if we make it~$K_r$-free by deleting a carefully chosen set of edges (\lemref{lem:sparse} in \secref{sec:sparse}). Together with \thmref{thm:universality}, this implies the existence of locally sparse tree-universal graphs.
\begin{theorem}
\label{thm:sparse}
There exists an absolute constant~$c\in\mathbb{R}^+$ such that the following statement holds. Let~$n\in\mathbb{N}$ and let~$r\in\mathbb{N}$ with~$r\ge 5$. Then there exists a graph with clique number at most~$r$ that is universal for $\mathcal{T}(n,c n^{1/3-2/(r+2)}/\log n)$.
\end{theorem}
In particular, \thmref{thm:sparse} implies that there exists a $\mathcal{T}(n,c n^{1/21}/\log n)$-universal graph with clique number at most five for all~$n\in\mathbb{N}$ if~$c$ is a small enough constant.

\subsection{Lower Bound Constructions}
In \thmref{thm:universality} we gave an \emph{upper} bound of $7\Delta n^{2/3}\log n$ on the minimum value~$d^*$ such that, for all $d\ge d^*$, every~$(n,d)$-expanders are $\mathcal{T}(n,d)$-universal. We now discuss \emph{lower} bounds on~$d^*$, \ie, constructions of $(n,d)$-expanders with (relatively) large values of~$d$ which are not universal for~$\mathcal{T}(n,\Delta)$. 

For random graphs, Krivelevich~\cite{Kr10} showed that, for every~$\eps>0$, there exists a~$\delta>0$ such that if~$n^{\eps}\le\Delta\le\frac{n}{\log n}$, then there exists a tree~$T\in\mathcal{T}(n,\Delta)$ such that the random graph~$G(n,p)$ with $p n=\delta\Delta\log n$  a.a.s.\ does not contain a copy of~$T$. In contrast, \thmref{thm:gnp} shows that there exist an absolute constant~$c$ such that $p n= c \Delta n^{2/3}\log n$ is sufficient for $G(n,p)$ to become universal for $\mathcal{T}(n,\Delta)$. This huge gap of order~$n^{2/3}$ seems to be  mainly an artifact of the proof of \thmref{thm:universality}.

In contrast to the random graph setting, we can show that, for $\Delta\in n^{\Omega(1)}$, the smallest value of~$d$ necessary for every~$(n,d)$-expander to be $\mathcal{T}(n,\Delta)$-universal grows faster than~$\Omega(\Delta\log n)$. To this end, recall that the radius of a connected graph~$G$ is defined as
\[
r(G):=\min_{u\in V}\max_{v\in V}\mathrm{dist}(u,v)
\]
where the \emph{distance} $\mathrm{dist}(u,v)$ between two vertices~$u$ and~$v$ is the length of a shortest path connecting~$u$ and~$v$ in~$G$. For example, the radius of the star-graph~$K_{1,n}$ is~$1$ and the radius of a path of even length~$2k$ is~$k$. A crucial observation is that we cannot embed a spanning graph onto a host graph with a strictly larger radius: the embedding itself would be a proof that the host graph has small radius, too. We will show that~$\mathcal{T}(n,\Delta)$ contains trees with a relatively small radius, whereas there are $(n,d)$-expanders (with quite large~$d$) with a fairly large radius. 

Now, consider the complete $(\lfloor\Delta\rfloor-1)$-ary tree on $n$ vertices  with~$\Delta\ge 3$, \ie, the rooted $(\lfloor\Delta\rfloor-1)$-ary tree in which every level, except possibly the last, is completely filled. This tree is in~$\mathcal{T}(n,\Delta)$ and its radius is strictly smaller than $1+\log n/\log(\lfloor\Delta\rfloor-1)$. On the other hand, we show below that there exist very strong expanders with radius at least~$1+\log n/\log(\lfloor\Delta\rfloor-1)$.
\begin{definition}[$\mathcal{G}_H$]
Let~$n\in\mathbb{N}$ and let~$H$ be a graph on~$n$ vertices. We define~$\mathcal{G}_H$ to be the class of $2n$-vertex graphs obtained from~$H$ by replacing each vertex~$v\in V(H)$ by two vertices~$u_v$ and~$u'_v$ and each edge~$\{v,w\}\in E(H)$ by either the two edges~$\{u_v,u_w\}$ and~$\{u'_v,u'_w\}$ or the two edges~$\{u_v,u'_w\}$ and~$\{u'_v,u_w\}$.
\end{definition}
For the class~$\mathcal{G}_H$, we will show the following result in \secref{sec:sparse}.
\begin{lemma}
\label{lem:radius}
There exists an absolute constant~$c\in\mathbb{R}^+$ such that the following statement holds. Let~$n,r\in\mathbb{N}$ satisfy~$c n^{1/r}\ge 3\log n$. Then there exists an $n$-vertex graph~$H$ such that all graphs in~$\mathcal{G}_H$ have radius at least~$r+2$ and a graph chosen uniformly at random from~$\mathcal{G}_H$ is a.a.s.\ a $(2n,c n^{1/r}\log^{-1}n)$-expander.
\end{lemma}
Combining \lemref{lem:radius} with our discussion of the radius of the complete $(\lfloor\Delta\rfloor-1)$-ary tree, we get the following theorem.
\begin{theorem}
\label{thm:lower}
There exists an absolute constant~$c\in\mathbb{R}^+$ such that the following statement holds. Let~$n,r\in\mathbb{N}$ and let~$\Delta:=n^{1/(r+1)}+2$ satisfy~$\Delta\ge c^{-1}\log n$. Then there exists an $(n,c \Delta^{1+1/r}\log^{-1} n)$-expander which is not $\mathcal{T}(n,\Delta)$-universal.
\end{theorem}
Note that although all graphs in~$\mathcal{G}_H$ have an even number of vertices we do not require this restriction in \thmref{thm:lower}: If we duplicate a vertex of an $(n,d)$-expander with radius~$r$ and connect the duplicate to all of the vertex's neighbors we obtain an $(n+1,d/2)$-expander which still has radius~$r$. 

\thmref{thm:lower} implies that there even exist $(n,c n/\log n)$-expanders which are not universal for $\mathcal{T}(n,(1+o(1))\sqrt{n})$. In comparison, we do not expect the same to be true a.a.s.\ for~$G(n,7c)$, which (a.a.s.) is a canonical example of an~$(n,c n/\log n)$-expander. Moreover, our construction complements a result of Koml{\'o}s, S{\'a}rk{\"o}zy, and Szemer{\'e}di~\cite{KoSaSz01}. They showed that, for every given $\delta$, there exists a constant $c:=c(\delta)$ such that every graph of minimum degree $(1+\delta)n/2$ is $\mathcal{T}(n,c n/\log n)$-universal. It is clear that this bound is sharp, since if we allow the minimum degree to be at most $\lfloor n/2\rfloor-1$, then the host graph may be disconnected. However, our $(\lfloor n/2\rfloor-1)$-regular construction shows that even host graphs that have an edge between all disjoint pairs of vertex sets of relatively small size~$O(\log n)$ may be not $\mathcal{T}(n,(1+o(1))\sqrt{n})$-universal. Finally, we remark that B{\"o}ttcher, Taraz, and W{\"u}rfl~\cite{BoTaWu11} observed a similar effect in the context of $(\eps,\delta)$-regular graphs and independently proposed a construction with similar properties as~$\mathcal{G}_H$.

\subsection{Universality for Almost All Spanning Trees}
For~$n\in\mathbb{N}$, let~$\mathcal{T}_n$ be the family of all labeled trees on the vertex set~$\{1,\dots,n\}$. Bender and Wormald~\cite{BeWo88} showed that for every fixed constant~$p\in (0,1)$ there is a subfamily~$\mathcal{T}^*_n\subseteq \mathcal{T}^*_n$ with $|\mathcal{T}^*_n|=(1-o(1))|\mathcal{T}_n|$ such that the random graph~$G(n,p)$ is a.a.s.\ universal for~$\mathcal{T}^*_n$. Note that this notion differs substantially from the (weaker) notion of being \emph{almost-universal} for~$\mathcal{T}_n$ (see, e.g.~\cite{FrKr06,Ca10}), which means that if~$T$ is drawn uniformly at random from all trees of~$\mathcal{T}_n$, then~$G$ a.a.s.\ contains a copy of~$T$.

It is well known (see, e.g.~\cite{Mo68}) that a tree chosen uniformly at random from~$\mathcal{T}_n$ has a.a.s.\ maximum degree at most $(1+o(1))\log n/\log \log n$. Therefore, the subfamily~$\mathcal{T}_n^*$ of all trees in~$\mathcal{T}_n$ with maximum degree at most $2\log n/\log \log n$ satisfies~$|\mathcal{T}_n^*|=(1-o(1))|\mathcal{T}_n|$. Thus, \thmref{thm:gnp} strengthens the result of Bender and Wormald as it allows us to replace the constant~$p\in(0,1)$ with a function $p\in o(1)$.
\begin{theorem}
\label{thm:almostuniversal}
There exists an absolute constant~$c\in\mathbb{R}^+$ and a subfamily~$\mathcal{T}^*_n\subseteq \mathcal{T}^*_n$ with $|\mathcal{T}^*_n|=(1-o(1))|\mathcal{T}_n|$ for every~$n\in\mathbb{N}$ such that the random graph~$G(n,p)$ with~$p\ge c n^{-1/3}\log^2 n$ is a.a.s.\ universal for~$\mathcal{T}^*_n$.
\end{theorem}

\subsection{The Maker-Breaker Game}
In recent studies of extremal properties of random graphs (like tree-universality), a central concept is that of \emph{robustness}. This means that we require some property to be not only \emph{typically} present in the respective graph class (\ie, to appear a.a.s.\ in a random graph drawn from this class), but also to persist after a modification of the random instance. There are numerous ways to model robustness, for example by the notion of resilience or via positional games. Here, we study the robustness of tree-universality in expanders in the setting of a Maker-Breaker game.
 
An $(a\,{:}\,b)$ Maker-Breaker game is played on a finite hypergraph~$(X,\mathcal{F})$ between two players, \emph{Maker} and \emph{Breaker}. The vertex set of the hypergraph is the \emph{board} and the hyperedges are the \emph{winning sets} in the game. The game is played in turns, starting with Maker's turn. In each of their turns, Maker claims~$a$ and Breaker claims~$b$ previously unclaimed vertices. The numbers~$a$ and~$b$ are called the \emph{biases} of Maker and Breaker, respectively. Maker's objective is to claim all elements of a winning set by the end of the game. In this case, Maker wins the game. Breaker's objective is to claim at least one element in each winning set by the end of the game. In this case, Breaker wins the game. The game ends when all vertices have been claimed, by which time either Maker or Breaker have won. 

We say that an $(a\,{:}\,b)$ Maker-Breaker game on~$(X,\mathcal{F})$ is \emph{Maker's win} if Maker has a strategy that allows him to win the game regardless of Breaker's strategy, otherwise the game is \emph{Breaker's win}. Clearly, every $(a\,{:}\,b)$ Maker-Breaker game $(X,\mathcal{F})$ is either Maker's or Breaker's win and the decision which of the two holds depends only on the parameters $a$, $b$, $X$, and~$\mathcal{F}$. For a more detailed discussion, we refer to~\cite{Be08}.

We now formulate a Maker-Breaker game for preserving tree-universality of a graph. Given a graph, Maker tries to claim a set of edges which induces a $\mathcal{T}(n,\Delta)$-universal subgraph.
\begin{definition}[Maker-Breaker Tree-Universality Game]
\label{def:mbuniversal}
For~$n\in\mathbb{N}$ and~$\Delta\in\mathbb{R}^+$, the \emph{Maker-Breaker $\mathcal{T}(n,\Delta)$-universality game} on a graph~$G$ is the Maker-Breaker game on the hypergraph $(E(G),\mathcal{F})$, where $\mathcal{F}$ consists of all edge sets~$F\subseteq E(G)$ such that the subgraph $(V(G),F)$ is $\mathcal{T}(n,\Delta)$-universal.
\end{definition}

Our main finding in this section is a condition for Maker's win in the $(1\,{:}\,b)$ Maker-Breaker Tree-Universality game on an $(n,d)$-expander.
\begin{theorem}
\label{thm:mbuniversal}
There exists an absolute constant~$c\in\mathbb{R}^+$ such that the following statement holds. Let~$n,b\in\mathbb{N}$ and let~$\Delta\in\mathbb{R}^+$ satisfy $\Delta\ge\log n$. Then the $(1\,{:}\,b)$ Maker-Breaker $\mathcal{T}(n,\Delta)$-universality game is Maker's win on every $(n,d)$-expander with $d\ge c b \Delta n^{2/3}\log n$.
\end{theorem}
\thmref{thm:mbuniversal} implies that, for~$\Delta\le\log n$ and in particular for $\Delta\in O(1)$ and~$n$ sufficiently large, the $(1\,{:}\,b)$ Maker-Breaker $\mathcal{T}(n,\Delta)$-universality game is Maker's win on every $(n,d)$-expander with $d\ge c b n^{2/3}\log^2 n$.

The proof of \thmref{thm:mbuniversal} is given in \secref{sec:makerbreaker}. Together with \lemref{lem:gnp}, it implies the following condition for Maker's win in the tree-universality game on binomial random graphs. 
\begin{corollary}
\label{cor:makerbreakergnp}
There exists an absolute constant~$c\in\mathbb{R}^+$ such that the following statement holds. Let~$b\colon\mathbb{N}\to\mathbb{N}$ and let $\Delta\colon\mathbb{N}\to\mathbb{R}^+$ satisfy $\Delta\ge\log n$. Then the $(1\,{:}\,b)$ Maker-Breaker $\mathcal{T}(n,\Delta)$-universality game is a.a.s.\ Maker's win on the random graph $G(n,p)$, provided that $p n\ge c b \Delta n^{2/3}\log^2 n$.
\end{corollary}

Correspondingly, \thmref{thm:mbuniversal} and \lemref{lem:regular} imply the following conditions for Maker's win for the tree-universality game on binomial random graphs. 
\begin{corollary}
\label{cor:makerbreakerregular}
There exists an absolute constant~$c\in\mathbb{R}^+$ such that the following statement holds. Let $b\colon\mathbb{N}\to\mathbb{N}$ and let $\Delta\colon\mathbb{N}\to\mathbb{R}^+$ satisfy $\Delta\ge\log n$. Then the $(1\,{:}\,b)$ Maker-Breaker $\mathcal{T}(n,\Delta)$-universality game is a.a.s.\ Maker's win on the random $r$-regular graph, provided that $r\ge c b \Delta n^{2/3}\log^2 n$, $r\in o(n)$, and~$r n$ is even.
\end{corollary}

\section{Properties of (n,d)-Expanders}
\label{sec:expander}
We now present all properties of $(n,d)$-expanders that are needed to prove the universality results in the remainder of this work. First, we observe that the expansion properties given in \defref{def:expander} are monotone in~$d$.
\begin{lemma}
\label{lem:monotone}
Let~$n\in\mathbb{N}$ and~$d,d_0\in\mathbb{R}^+$ satisfy~$3\le d_0\le d\le n/6$. Then every $(n,d)$-expander is also an $(n,d_0)$-expander.
\end{lemma}

\begin{proof}
Let~$m:=m(n,d)$ and~$m_0:=m(n,d_0)$. Since~$m_0\ge m$, condition~(\eb) holds immediately for the parameter~$m_0$, and since~$d_0\le d$, condition~(\ea) holds immediately for the parameters~$m$ and~$d_0$. Thus, it is sufficient to verify that $|N_G(X)|\ge d_0|X|$ holds for all $X\subseteq V(G)$ with $m\le |X|< m_0$. For such a set~$X$, we have by condition~(\eb) that
\[
|N_G(X)|\ge n-|X|-m\ge 2 d_0(m_0-1)-2m_0\ge d_0 m_0\ge d_0|X|.
\]
The lemma follows.
\end{proof}

Next, as a direct consequence of \defref{def:expander}, we give a lower bound on the number of edges between two large disjoint sets.
\begin{lemma}
\label{lem:density}
Let~$m\in\mathbb{N}$. Let~$G$ be a graph such that $e_G(X,Y)>0$ holds for all disjoint $X,Y\subseteq V(G)$ with $|X|=|Y|= m$. Then
\[
e_G(X,Y)\ge\frac{|X||Y|}{4m}
\]
holds for all disjoint $X,Y\subseteq V(G)$ with $|X|\ge m$ and~$|Y|\ge 2m$.
\end{lemma}

\begin{proof}
Partition~$X$ into~$k:=\lfloor\frac{|X|}{m}\rfloor\ge\frac{|X|}{2m}$ disjoint parts~$X_1,\dots,X_k$, each of size at least~$m$. By the prerequisite of the lemma, we have~$|N_G(X_i)\cap Y|\ge |Y|-m\ge |Y|/2$ for all~$i\in\{1,\dots,k\}$. Thus, $e_G(X,Y)\ge k|N_G(X_i)\cap Y|$ and \lemref{lem:density} follows.
\end{proof}

The following fact is an important insight into the structure of sparse expanders and is frequently used in the proof of \thmref{thm:universality}. It allows us to bound the number of vertices with small neighborhood in any sufficiently large vertex set of an expander.
\begin{lemma}[Small Exceptional Sets]
\label{lem:exceptionalset}
Let~$G$ be a graph, let~$m\in\mathbb{N}$, and let~$W\subseteq V(G)$ satisfy $|W|\ge m^2$. We call a vertex in~$V(G)\setminus W$ \emph{exceptional} with respect to~$W$ and~$m$ if it has at most $m-1$ neighbors in~$W$. Suppose that $e_G(X,Y)>0$ for all~$X\subseteq V(G)\setminus W$ and all~$Y\subseteq W$ that satisfy $|X|=|Y|=m$. Then there are at most~$m-1$ exceptional vertices with respect to~$W$ and~$m$.
\end{lemma}
In the following proof a simple counting argument shows that if there exist~$m$ exceptional vertices in~$V(G)\setminus W$, then there are at least~$m$ vertices in~$W$ which are not in their neighborhood --- a contradiction to (\eb) of \defref{def:expander}.
\begin{proof}
Let~$V:=V(G)$. Assume for contradiction that there exists a set~$X\subseteq V\setminus W$ with~$|X|=m$ such that $|N_G(v)\cap W|<m$ for all~$v\in X$. 

On one hand, since~$|X|=m$ and therefore~$|V\setminus(X\cup N_G(X))|\le m-1$, we have
\[
e_G(X,W)\ge|W\cap N_G(X)|\ge |W|-(m-1)\ge m^2-m+1.
\]
On the other hand,
\[
e_G(X,W)=\sum_{x\in X}|N_G(x)\cap W|\le m(m-1)=m^2-m,
\]
which is clearly a contradiction. Thus, no such set~$X$ exists and there are at most~$m-1$ vertices in~$V\setminus W$ with fewer than~$m$ neighbors in~$W$.
\end{proof}

\subsection{Partitioning Expanders}
We now show that we can partition the vertex set of an $(n,d)$-expander in such a way  that the neighborhoods of small expanding sets distribute between the parts according to the sizes of the parts. In \secref{sec:universality}, this technique plays a major role in the proof of our main result, the tree-universality of sparse expanders (\thmref{thm:universality}). 
\begin{lemma}[Partition Lemma]
\label{lem:partition}
There exists an absolute constant~$n_0\in\mathbb{N}$ such that the following statement holds. Let~$k,n\in\mathbb{N}$ and~$d\in\mathbb{R}^+$ satisfy~$n\ge n_0$ and~$k\le\log n$.  Furthermore, let $n_1,\dots,n_k\in\mathbb{N}$ satisfy $n=n_1+\dots +n_k$ and let~$d_i:=\frac{n_i}{5n}d$ satisfy~$d_i\ge 2\log n$ for all~$i\in\{1,\dots,k\}$. 

Then, for every $(n,d)$-expander~$G$, the vertex set~$V(G)$ can be partitioned into~$k$ disjoint sets $U_1,\dots,U_k$ of sizes $n_1,\dots,n_k$, respectively, such that
\begin{equation}
\label{eq:partition}
|N_G(X)\cap U_i|\ge d_i|X|
\end{equation}
holds for all sets $X\subseteq V$ with $1\le|X|< m(n,d)$ and all $i\in\{1,\dots,k\}$. Moreover, the induced subgraph~$G[W_i]$ is a $(|W_i|,d_i)$-expander for all~$i\in\{1,\dots,k\}$ and all~$W_i\subseteq V(G)$ with $U_i\subseteq W_i$.
\end{lemma}

This statement can be shown using the probabilistic method: Using the union bound and a tail bound on the hypergeometric distribution, we show that a uniformly random partition of an $(n,d)$-expander into~$k$ parts of sizes~$n_1,\dots,n_k$ satisfies~(\ref{eq:partition}).

Before we prove \lemref{lem:partition}, we first state a well-known result (see, e.g.,~\cite[Theorem~2.10]{JaLuRu00}) for bounding the tail probabilities of the hypergeometric distribution~$\Hyp(n,m,\ell)$. A random variable~$X$ distributed according to $\Hyp(n,m,\ell)$ models the number of white balls found among~$\ell$ balls drawn without replacement from an urn containing~$n$ balls, $m$ of which are white. Recall that $\Pr(X=k)=\binom{m}{k}\binom{n-m}{\ell-k}/\binom{n}{\ell}$ for all~$0\le k\le n$ and that~$\EXP[X]=m\ell/n$.
\begin{theorem}
\label{thm:hyp}
Let~$\eps$ be a positive constant satisfying~$\eps\le 3/2$ and let~$X\sim\Hyp(n,m,\ell)$. Then
\[
\Pr\big[|X-\EXP[X]|>\eps\big]\le\euler^{-\frac{\eps^2}{3} \EXP[X]}.
\]
\end{theorem}

\begin{proof}[Proof of \lemref{lem:partition}]
Choose~$n_0$ such that $\log n\le n^{2/15}$. Let~$V=V(G)$ and~$m:=m(n,d)$. We show the existence of a partition~$U_1,\dots,U_k$ which respects (\ref{eq:partition}) by a simple probabilistic argument. 

Choose a partition~$U_1,\dots,U_k$ of~$V$ into disjoint sets of respective sizes $n_1,\dots,n_k$ uniformly at random. We show that with positive probability, (\ref{eq:partition}) holds for all sets $X\subseteq V$ with $1\le|X|< m$ and all $i\in\{1,\dots,k\}$.

Let~$X\subseteq V$ with $1\le |X|< m$ and let~$i\in\{1,\dots,k\}$. Then the random variable~$|N_G(X)\cap U_i|$ is distributed according to the hypergeometric distribution~$\Hyp(n,n_i,|N_G(X)|)$ with
\[
\EXP\big[|N_G(X)\cap U_i|\big]=\frac{n_i}{n}\,|N_G(X)|\ge \frac{n_i}{n}\,d\,|X|=5d_i|X|.
\]
We apply \thmref{thm:hyp} with~$\eps=4/5$ and obtain 
\[
\Pr\big[|N_G(X)\cap U_i|\le d_i|X|\big]\le \euler^{-\frac{16}{15}d_i|X|}\le e^{-\frac{32}{15}|X|\log n}=n^{-\frac{32}{15}|X|}.
\]

Let~$q$ be the probability that there exists a set $X\subseteq V$ with $1\le|X|< m$ and an~$i\in\{1,\dots,k\}$ which violates property~(\ref{eq:partition}). Then, by the union bound,
\[
q\le\sum_{i=1}^k\sum_{j=1}^m\binom{n}{j} n^{-\frac{32}{15}{j} }<\sum_{i=1}^k\sum_{j=1}^n n^j n^{-\frac{32}{15}{j} }\le k n^{-\frac{2}{15}
}\le 1
\]
for sufficiently large~$n$.

We have shown that with positive probability the randomly chosen partition~$U_1,\dots,U_k$ satisfies property~(\ref{eq:partition}); therefore, such a partition exists and the first statement of \lemref{lem:partition} holds.

Finally, let~$U_1,\dots,U_k$ be such a partition that satisfies property~(\ref{eq:partition}). Let~$i\in\{1,\dots,k\}$ and consider a set~$W\subset V$ with~$U_i\subset W$ and the induced graph~$H=G[W]$. Then, by the choice of~$d_i$, we have $m(|W|,d_i)\ge m(n_i,d_i)\ge m(n,d)$. Thus, condition~(\eb) in \defref{def:expander} with $m=m(|W|,d_i)$ holds for~$H$ since~$G$ is an $(n,d)$-expander. By~(\ref{eq:partition}), $|N_H(X)|\ge d_i|X|$ holds for all~$X\subseteq V(H)$ with~$1\le |X|< m(n,d)$. Thus, similar to the proof of \lemref{lem:monotone}, it is sufficient to verify that $|N_H(X)|\ge d_i|X|$ holds also for all $X\subseteq V(H)$ with $m(n,d)\le |X|< m(|W|,d_i)$. Since~$G$ is an $(n,d)$-expander, we have for such a set~$X$ that
\[
|N_H(X)|\ge |W|-|X|-m(n,d)\ge 2 d_i(m(n_i,d_i)-1)-2m(n_i,d_i)\ge d_i m(n_i,d_i)\ge d_i|X|
\]
and the second statement of \lemref{lem:partition} holds.
\end{proof}

\subsection{Almost Spanning Trees, Hamilton Paths, and Star Matchings}
We now summarize three known results on embedding almost spanning trees, Hamilton paths, and star matchings in graphs with large expansion. These results are crucial for the proof of \thmref{thm:universality} (Tree-Universality).

In~\cite{Ha01}, Haxell extended a result of Friedman and Pippenger~\cite{FrPi87} and showed that one can embed every almost spanning tree with bounded maximum degree in a graph with sufficiently large expansion. Here, we present a formulation of this result in the flavor of Theorem~3 in~\cite{BaCsPeSa10}.
\begin{theorem}
\label{thm:almostspanning}
Let~$d,m,k\in\mathbb{N}$ and let~$H$ be a non-empty graph satisfying the following two conditions:
\begin{enumerate}[(i)]
\item $|N_H(X)|\ge d|X|+1$ for all $X\subseteq V(H)$ with $1\le|X|\le m$,
\item $|N_H(X)|\ge d|X|+k$ for all $X\subseteq V(H)$ with $m<|X|\le 2m$.
\end{enumerate}
Then~$H$ contains a copy of every tree~$T$ with $|V(T)|\le k+1$ and maximum degree at most~$d$.
\end{theorem}

In terms of $(n,d)$-expanders, we may reformulate the previous theorem as follows.
\begin{corollary}[Almost Spanning Tree Embedding]
\label{cor:almostspanning}
Let~$n,\Delta\in\mathbb{N}$ and let~$d\in\mathbb{R}^+$ with~$d\ge 2\Delta$. Then every $(n,d)$-expander is $\mathcal{T}(n-4\Delta m(n,d),\Delta)$-universal.
\end{corollary}

\begin{proof}
Without loss of generality we may suppose that~$\Delta\ge 2$. Let~$m:=m(n,d)$ and let $k:=|V(T)|$. Furthermore, let~$H$ be an $(n,d)$-expander with $n=k+4\Delta m$ and let~$T\in\mathcal{T}(k,\Delta)$.

Then, for all~$X\subseteq V(H)$ with $1\le |X|\le m$, we have  by (\ea) that
\[
|N_H(X)|\ge 2\Delta|X|\ge \Delta |X|+1.
\]
For all~$X\subseteq V(H)$ with $m\le |X|\le 2m$, we have
\[
|N_H(X)|\ge n-|X|-m\ge k+4\Delta m-3m\ge k+2\Delta m\ge\Delta|X|+k.
\]
The corollary then follows from \thmref{thm:almostspanning}.
\end{proof}

Next, we state a result of Hefetz, Krivelevich, and Szab{\'o}~\cite{HeKrSz09} on the Hamilton-connectedness of expanders with edge-connectivity between large sets. For this, let us briefly revisit the notion of \emph{Hamilton-connectedness}.  An \emph{$x$-$y$-Hamilton path} in a graph is a path with end-vertices~$x$ and~$y$ that visits each vertex of the graph exactly once. A graph is \emph{Hamilton-connected} if there exists an $x$-$y$-Hamilton path for every pair of vertices~$x$ and~$y$ in the graph. The following theorem is a simplified version of the results in~\cite{HeKrSz09}.
\begin{theorem}
\label{thm:hamilton}
Let~$n,d\in\mathbb{N}$ satisfy that~$n$ is sufficiently large and~$12\le d\le\sqrt{n}$. Let~$H$ be a graph on~$n$ vertices satisfying the following two conditions:
\begin{enumerate}[(i)]
\item  $|N_H(X)|\ge d|X|$ for all $X\subseteq V(H)$ with $0<|X|\le\frac{n\log d}{d\log n}$,
\item  $e_H(X,Y)>0$ for all disjoint $X,Y\subseteq V(H)$ with $|X|=|Y|\ge\frac{n\log d}{1035\log n}$.
\end{enumerate}
Then~$H$ is Hamilton-connected.
\end{theorem}

As before, we give a reformulation of this result in terms of $(n,d)$-expanders.
\begin{corollary}[Hamilton Connectivity]
\label{cor:hamilton}
There exists an absolute constant~$n_0\in\mathbb{N}$ such that the following statement holds. Let~$n\in\mathbb{N}$ with~$n\ge n_0$ and let~$d\in\mathbb{R}^+$ with $d\ge\log n$. Then every $(n,d)$-expander is Hamilton-connected.
\end{corollary}

\begin{proof}
Let~$G$ be an $(n,d)$-expander. For $d_{\text{Thm~\ref*{thm:hamilton}}}=\euler^{1035}$, consider the two conditions of \thmref{thm:hamilton}. Condition~(i) holds by \lemref{lem:monotone} for sufficiently large~$n$. Condition~(ii) holds since
\[
m\le\frac{n}{d}\le\frac{n}{\log n}=\frac{n\log d_{\text{Thm~\ref*{thm:hamilton}}}}{1035\log n}.
\]
Thus, $G$ is Hamilton-connected by \thmref{thm:hamilton}.
\end{proof}

Finally, we state a version of Hall's marriage theorem for expanders which shows that we can embed a star matching in a bipartite graph with large expansion in one direction and large minimum degree in the other direction.
\begin{lemma}[Star Matching]
\label{lem:starmatching}
Let~$d,m\in\mathbb{N}$  and let~$G$ be a graph. Suppose that two disjoint sets~$U,W\subseteq V(G)$ satisfy the following three conditions:
\begin{enumerate}[(i)]
\item $|N_G(X)\cap W|\ge d|X|$ for all $X\subseteq U$ with $1\le |X|\le m$,
\item $e_G(X,Y)>0$  for all $X\subseteq U$ and~$Y\subseteq W$ with $|X|=|Y|\ge m$,
\item $|N_G(w)\cap U|\ge m$ for all $w\in W$.
\end{enumerate}
Then, for every map~$k\colon U\to\{0,\dots,d\}$ that satisfies~$\sum_{u\in U} k(u)=|W|$, the set~$W$ can be partitioned into~$|U|$ disjoint subsets~$\{W_u\}_{u\in U}$ satisfying~$|W_u|=k(u)$ and~$W_u\subseteq N_G(u)\cap W$. We call the set of edges between the vertices of~$U$ and their respective parts in~$W$ a \emph{star matching}.
\end{lemma}

\begin{proof}
To prove this lemma, we show for all~$X\subseteq U$ the generalized Hall's condition,
\begin{equation}
\label{eq:hall}
|N_G(X)\cap W|\ge\sum_{x\in X}k(x).
\end{equation}
We distinguish three cases: 

First, if $|X|< m$, then~$k(x)\le d$ for all~$x\in X$ and~(i) implies~(\ref{eq:hall}).

Second, if~$m\le |X|\le |U|-m$, then~$k(u)\ge 1$ for all~$u\in U\setminus X$ and~(ii) implies~(\ref{eq:hall}). 

Third, if~$|U|-m<|X|$, then~(iii) directly implies~(\ref{eq:hall}).

\noindent Thus,~(\ref{eq:hall}) holds for all~$X\subseteq U$ and the lemma is a direct consequence of the Max-Flow Min-Cut Theorem~\cite{ElFeSh56,FoFu56}.
\end{proof}

\section{Tree-Universality of (n,d)-Expanders}
\label{sec:universality}
This section is devoted to the proof of our main result, \thmref{thm:universality} (Tree-Universality), which we presented in the introduction. The proof is based on a case distinction on whether the embedded tree contains a long bare path or many leaves. This extends the ideas in~\cite{Kr10}.

\begin{definition}[Leaves, Bare Paths, and Levels]
Let~$T$ be a tree. A \emph{leaf} of~$T$ is a vertex of degree one in~$T$. A \emph{bare path} is a path in~$T$ whose vertices have all degree two in~$T$. If we remove all leaves from~$T$, we call the leaves and bare paths in the remaining tree \emph{second level leaves} and \emph{second level bare paths}, respectively.  For distinction, we call the leaves and bare paths of the original tree~$T$ also \emph{first level leaves} and \emph{first level bare paths}. 
\end{definition}

The following observation was already made in~\cite{Kr10} and states that a tree with bounded maximum degree contains a long bare path or many leaves.

\begin{lemma}
\label{lem:pathorleaves}
Let~$T$ be a tree, let $P$ be a bare path of maximum length in~$T$, and let~$L$ be the set of leaves in~$T$. Then
\[
2\big(|V(P)|+1\big)\big(|L|-1\big)\ge|V(T)|.
\]
\end{lemma}

\begin{proof} Let~$V:=V(T)$, let~$L:=\{v\in V\mid\deg_T(v)=1\}$ and let $B:=\{v\in V\mid\deg_T(v)\ge 3\}$. Then we have $|B|\le |L|-2$, since
\[
-2=2|E(T)|-2|V(T)|=\sum_{v\in V}\big(\deg_T(v)-2\big)\ge |B|-|L|.
\]
Next, we root~$T$ at an arbitrary leaf. This allows us to injectively map the bare paths of~$T$ to the set~$L\cup B$ by assigning every bare path to the leaf or branching vertex adjacent to it farther away from the root. Therefore, the number of bare paths is at most~$|L|+|B|$. Since every vertex in~$V$ is either in~$L$, in~$B$, or in a bare path of~$T$, this implies that $|V|\le |L|+|B|+|V(P)|(|L|+|B|)$ and therefore the lemma.
\end{proof}

Consider the setting of \thmref{thm:universality}. Let~$c\in\mathbb{R}^+$ be sufficiently small and assume that~$n$ is sufficiently large. Let $T\in\mathcal{T}(n,\Delta)$ with $\log n\le\Delta\le c n^{1/2}$, and let~$G$ be an $(n,d)$-expander with
\[
d:=7\Delta n^{2/3}.
\]
Recall that $m:=m(n,d)=\lceil\frac{n}{2d}\rceil$. \lemref{lem:pathorleaves} tells us that~$T$ contains a bare path on~$50\Delta m$ vertices or has at least~$25\Delta m^2$ leaves, since $2(50\Delta m+1)(25\Delta m^2-1)< n$ for sufficiently large~$n$ and small~$c$. In fact, $2(50\Delta m+1)(25\Delta m^2-1)< n/\Delta$ for sufficiently large~$n$. Therefore, if $L$ is the set of leaves in~$T$, then $T-L$ still contains a bare path on~$50\Delta m$ vertices or has at least~$25\Delta m^2$ leaves, since~$|T-L|\ge n/\Delta$. Based on this observation, we consider three cases.

\paragraph{Case 1.} \emph{$T$ contains a first level bare path on at least~$50\Delta m$ vertices.}

In this case, we use \corref{cor:almostspanning} (Almost Spanning Tree Embedding) to first embed all of~$T$ in~$G$ except for the bare path (whose removal splits~$T$ into two rooted trees). Then we apply \corref{cor:hamilton} (Hamilton Connectivity) to also embed the bare path by connecting the two roots by a path covering all the unused vertices in~$G$. The details of this argument are given in \propref{prop:barepath}.

\paragraph{Case 2.} \emph{$T$ has at least~$25\Delta m^2$ first level leaves and contains a second level bare path on at least~$50\Delta m$ vertices.}

In this case,~$T$ has many leaves. We use \corref{cor:almostspanning} (Almost Spanning Tree Embedding) to first embed all of~$T$ in~$G$ except for the second level bare path and the leaves. Then we use \corref{cor:hamilton} (Hamilton Connectivity) to embed the bare path. Finally, we use \lemref{lem:starmatching} (Star Matching) to embed the leaves of~$T$. Note that once~$T$ without the leaves is embedded, we know which vertices in~$G$ are the images of the parents of the leaves of~$T$. We call these vertices in~$G$ the \emph{portals} of the leaves. In order to embed the leaves, we need to find a star matching in~$G$ between the set of portals and the set of vertices which remain free after the embedding of~$T$ without the leaves. However, if we are not careful, then after embedding~$T$ without the leaves (and thus fixing the set of portals), some of the remaining vertices of~$G$ may be not connected to any of the portals. As these vertices would prevent us from finding a star matching, we call them \emph{exceptional} vertices. We solve this problem by forcing the second level bare path to cover all exceptional vertices. The details of this argument are given in \propref{prop:leavepath}.

\paragraph{Case 3.} \emph{$T$ has at least~$25\Delta m^2$ first level leaves and at least~$25\Delta m^2$ second level leaves.}

In this case, $T$ has many (first level) leaves that are attached to second level leaves. We again use \corref{cor:almostspanning} (Almost Spanning Tree Embedding) to embed~$T$ without these two levels of leaves and then embed the leaves of each level separately using \lemref{lem:starmatching} (Star Matching). Here, there again may exist a set of exceptional vertices which can spoil the embedding of the second level leaves. We apply a similar argument as in Case~2, only this time the original leaves of~$T$ take the role that the second level bare path played before, \ie, cover the set of exceptional vertices. The details for this argument are given in \propref{prop:multileave}.

\paragraph{}In the remainder of this section we show three results (\propref{prop:barepath}, \propref{prop:leavepath}, and \propref{prop:multileave}) which cover the cases discussed above. \thmref{thm:universality} (Tree-Universality) is a direct consequence of these three propositions.

\begin{proposition}[Case 1]
\label{prop:barepath}
The statement of \thmref{thm:universality} holds with $\mathcal{T}(n,\Delta)$ restricted to trees that contain a first level bare path on at least~$50\Delta m$ vertices.
\end{proposition}

\begin{proof}
We construct an embedding~$\phi$ of~$T$ onto~$G$ as follows. We split~$T$ into two parts and embed them consecutively. These parts are a (first level) bare path~$P$ on exactly~$50\Delta m$ vertices (chosen as a subpath of a longest bare path in~$T$) and the remaining forest~$F:=T[V(T)\setminus V(P)]$ on $n-50\Delta m$ vertices which consists of two trees. Note that~$|V(F)|\ge |V(P)|$ since~$50\Delta m\le n/3$ for sufficiently large~$n$. Let~$s_P$ and~$t_P$ be the two end-vertices of~$P$ and let~$s_F$ and~$t_F$, respectively, be their two neighbors in~$F$.

We construct an embedding~$\phi$ of~$T$ in~$G$ in two steps. The first step is to find an embedding~$\phi_F$ of the forest~$F$ in~$G$ and the second step is to find an embedding~$\phi_P$ of the path~$P$ in~$G[V\setminus\phi_F(F)]$. In this, we make sure that these embeddings satisfy $\{\phi_P(s_P),\phi_F(s_F)\}\in E(G)$ and $\{\phi_P(t_P),\phi_F(t_F)\}\in E(G)$.
 
We start by partitioning~$V$ into~$U_F$ and~$U_P$ which (partially) host the embeddings of~$F$ and~$P$. For this, we apply \lemref{lem:partition} (Partition Lemma) to partition~$V$ into two sets~$U_F$ and~$U_P$ with $|U_F|=|V(F)|+4\Delta m$ and $|U_P|=|V(P)|-4\Delta m$. Note that~$U_F$ and~$U_P$ are each of size at least $20\Delta m$. Since
\[
\frac{|U_P|}{5n}\,d\ge\frac{|U_P|}{10m}\ge 2\Delta~(\ge 2\log n),
\]
the prerequisites of \lemref{lem:partition} are satisfied. Thus, $G[U_F]$ is a $(|U_F|,2\Delta)$-expander and also $G[W_P]$ is a $(|W_P|,2\Delta)$-expander for every set~$W_P$ with~$U_P\subseteq W_P\subseteq V$.

Now, we turn to the actual constructions of~$\phi_F$ and~$\phi_P$. First, we determine~$\phi_F$. By \corref{cor:almostspanning} (Almost Spanning Tree Embedding), there exists an embedding~$\phi_F$ of~$F$ in~$G[U_F]$. Note that since this result only allows us to embed almost spanning trees, $U_F$ was chosen to be somewhat larger than~$|V(F)|$.

Next, we move the unused~$4\Delta m$ vertices of~$U_F$ to~$U_P$ and embed~$P$ by applying \corref{cor:hamilton} (Hamilton Connectivity). Let~$W_P:=V\setminus \phi_F(F)$. Since~$U_P\subseteq W_P$, we already know that $G[W_P]$ is a~$(50\Delta m,2\Delta)$-expander. Moreover, by the properties given in \lemref{lem:partition} (Partition Lemma), we can find two (distinct) vertices~$v\in N_G(\phi_F(s_F))\cap U_P$ and~$w\in N_G(\phi_F(t_F))\cap U_P$ to which we embed~$s_P$ and~$t_P$, respectively. Since~$2\Delta\ge\log n$, $G[W_P]$ contains a Hamilton path connecting~$v$ and~$w$ by \corref{cor:hamilton} (Hamilton Connectivity). Let~$\phi_P$ be the embedding of the bare path~$P$ to this Hamilton path. This concludes the construction of~$\phi$ and the proof of the proposition.
\end{proof}

\begin{proposition}[Case 2]
\label{prop:leavepath}
The statement of \thmref{thm:universality} holds with $\mathcal{T}(n,\Delta)$ restricted to trees that have at least~$25\Delta m^2$ first level leaves and contain a second level bare path on at least~$50\Delta m$ vertices.
\end{proposition}

\begin{proof}
Let~$L\subseteq V(T)$ be the set of (first level) leaves of~$T$. By the assumptions of the proposition, we have $|L|\ge 25\Delta m^2$. Let~$K$ be the set of neighbors of the leaves in~$T$, \ie,~$K=N_T(L)$.

Since~$T$ contains a second level bare path on at least~$50\Delta m$ vertices, we can find two vertex-disjoint second level bare paths on exactly~$25\Delta m$ vertices in~$T$ (e.g., two subpaths of a longest bare path). Of these two second level bare paths, let~$P$ be the one that contains \emph{at most} $|K|/2$ vertices of~$K$. 

Let~$F$ be the forest~$T[V(T)\setminus(L\cup V(P))]$; note that $F$ consists of two trees. Like in \propref{prop:barepath}, let~$s_P$ and~$t_P$ be the end-vertices of~$P$ and let~$s_F$ and~$t_F$ be their respective neighbors in~$F$.

We partition~$K$ into~$K_F:=K\cap V(F)$ and~$K_P:=K\cap V(P)$. Since each vertex in~$K$ is adjacent to at most~$\Delta$ vertices in~$L$, we have $|K|\ge 25 m^2$. Moreover, $|K_F|\ge| K_P|$ by the choice of~$P$ and therefore
\begin{equation}
\label{eq:manyforestportals}
|K_F|\ge m^2.
\end{equation}

Similarly as in \propref{prop:barepath}, we construct an embedding~$\phi$ of~$T$ in~$G$ in several steps. First, we construct an embedding~$\phi_F$ of the forest~$F$, then an embedding~$\phi_P$ of the path~$P$, and finally an embedding~$\phi_L$ of the leaves~$L$. In this, we make sure that the images of~$s_P$~and~$t_P$ are adjacent to those of~$s_F$ and~$t_F$, respectively, and that the image of each leaf in~$L$ is adjacent to the image of its respective neighbor in~$K$.
 
Again, we partition~$G$ into sets which partially host the embeddings of~$F$,~$P$, and~$L$. For this, we apply \lemref{lem:partition} (Partition Lemma) to partition~$V$ into the three parts~$U_F$, $U_P$, and~$U_L$ satisfying~$|U_F|=|V(F)|+4\Delta m$, $|U_L|=|L|+\Delta m$, and~$|U_P|=|V(P)|-5\Delta m$. Then~$U_F$ and~$U_P$ are each of size at least $20\Delta m$ and~$U_L$ is of size at least~$20\Delta m^2$. Hence, the subgraph $G[U_F]$ is a $(|U_F|,2\Delta)$-expander and also $G[W_P]$ is a $(|W_P|,2\Delta)$-expander for every set~$W_P$ with~$U_P\subseteq W_P\subseteq V$. Moreover, for every set~$X\subseteq V$ with~$1\le |X|< m$, it holds that
\[
|N_G(X)\cap U_L|\ge 2\Delta m |X|
\]
and therefore we have, for every set~$W_L$ of size~$|L|$ with $W_L\subseteq U_L$, that
\begin{equation}
\label{eq:lexpansion}
|N_G(X)\cap W_L|\ge |N_G(X)\cap U_L|-\Delta m\ge \Delta |X|
\end{equation}
for all sets~$X\subseteq V\setminus W_L$ with~$1\le |X|< m$.

In order to construct~$\phi$, we first apply \corref{cor:almostspanning} (Almost Spanning Tree Embedding) to find an embedding~$\phi_F$ of~$F$ in~$U_F$. Let~$W_F:=\phi(F)$. Later, we move the remaining~$4\Delta m$ vertices in~$U_F\setminus W_F$ to~$U_P$. 

Next, we embed the second level bare path~$P$. However, before doing so, we identify the exceptional set of vertices~$Z\subseteq U_L$ which might later spoil the application of \lemref{lem:starmatching} (Star Matching) for the embedding of~$L$. This set, denoted by~$Z$, contains all vertices in~$U_L$ that have fewer than~$m$ neighbors in~$\phi(K)$, the set of portals. At this point of the construction of~$\phi$, we only know~$\phi(K_F)$, which is equal to~$\phi_F(K_F)$. However, since~$|K_F|\ge |K_P|$ and therefore~$|\phi_F(K_F)|\ge |\phi(K)|/2$, we may already define~$Z$.

Let~$Z:=\{u\in U_L\mid |N_G(u)\cap\phi_F(K_F)|<m\}$. We already know that $|\phi(K_F)|=|K_F|\ge m^2$. Thus, since~$G$ is an $(n,d)$-expander, we have by \lemref{lem:exceptionalset} (Small Exceptional Sets) that $|Z|\le m\le\Delta m$. Let~$W_L\subseteq U_L$ be an arbitrary set of size~$|L|$ that contains no vertex in~$Z$. In the third step of the embedding, $W_L$ will be the image of~$L$ under~$\phi_L$. Note that however we embed~$K_P$, this choice of~$W_L$ ensures that~$|N_G(u)\cap\phi(K)|\ge m$ holds for every vertex~$u\in W_L$. In fact, the only reason why we separated the embedding of the second level bare path~$P$ from the embedding of~$F$ is to take care of the exceptional set~$Z$.

Now, we return to the embedding~$\phi_P$ of the second level bare path~$P$. So far, we constructed the embedding~$\phi_F$ of~$F$ to the set~$W_F\subseteq U_F$ of size~$|V(F)|$ and reserved the set~$W_L\subseteq U_L$ of size~$|L|$ for the embedding~$\phi_L$ of~$L$. Let~$W_P:=V\setminus(W_F\cup W_L)$. Then~$|W_P|=|P|$ and we have already seen that $G[W_P]$ is a $(|W_P|,2\Delta)$-expander. Moreover, as in the proof of the previous proposition, we can choose two distinct vertices~$v\in N_G(\phi_F(s_F))\cap U_P$ and~$w\in N_G(\phi_F(t_F))\cap U_P$ as the images of~$s_P$ and~$t_P$ in~$\phi_P$, respectively. Afterwards, we define~$\phi_P$ by embedding~$P$ onto a Hamilton path between~$v$ and~$w$ in~$G[W_P]$ given by \corref{cor:hamilton} (Hamilton Connectivity).

Finally, we construct an embedding~$\phi_L$ of~$L$ by applying \lemref{lem:starmatching} (Star Matching). At this point, the embedding of~$K$ is already given by the embeddings~$\phi_F$ of~$K_F$ and~$\phi_P$ of~$K_P$. Thus, it suffices to verify that the conditions of \lemref{lem:starmatching} are satisfied. Condition~(i) holds by~(\ref{eq:lexpansion}), condition~(ii) holds since~$G$ is an $(n,d)$-expander, and condition~(iii) holds since we excluded~$Z$ from~$U_L$ when choosing~$W_L$. Thus, we find an embedding~$\phi_L$ of~$L$ to~$G$ that respects the edges between~$K$ and~$L$ in~$T$. This concludes the construction of~$\phi$ and the proof of the proposition.
\end{proof}

\begin{proposition}[Case 3]
\label{prop:multileave}
The statement of \thmref{thm:universality} holds with $\mathcal{T}(n,\Delta)$ restricted to trees that have at least~$25\Delta m^2$ (first level) leaves and at least~$25\Delta m^2$ second level leaves.
\end{proposition}

\begin{proof}
Let~$L'$ be the set of first level leaves of~$T$ and~$M'$ be the set of second level leaves of~$T$, \ie, the set of leaves of~$T[V(T)\setminus L]$. By the assumptions of the proposition, $|L'|\ge 25\Delta m^2$ and~$|M'|\ge 25\Delta m^2$. Note that every second level leaf~$v\in M'$ has at least one first level leaf attached to it since otherwise~$v$ would have been in~$L'$ to begin with.

As in the proof of \propref{prop:leavepath}, we split~$T$ into three parts. Let~$M$ be an arbitrary subset of~$M'$ of size exactly~$25\Delta m^2$, let~$L:=N_T(M)\cap L'$ be the set of first level leaves with neighbors in~$M$, and let~$F$ be the induced subtree $F:=T[V(T)\setminus(L\cup M)]$. Let~$K:=N_{T-L}(M)$ be the neighbors of the second level leaves~$M$ in~$T$ without the first level leaves. Then~$|K|\ge 25 m^2$. Note that by definition~$M$ is the set of portals of~$L$, that is~$M=N_T(L)$. Also note that $|M\cup L|\le 25\Delta^2m^2$ and thus $|V(F)|\ge 16\Delta m$ for sufficiently large~$n$.

We construct an embedding~$\phi$ of~$T$ in~$G$ by defining three partial embeddings~$\phi_F$,~$\phi_M$, and~$\phi_L$ of~$F$, $M$, and~$L$, respectively. As before, we make sure that these embeddings respect the edges linking~$K$ to~$M$ and~$M$ to~$L$ in~$T$. A crucial step in this process will be again to handle the exceptional set~$Z$ that might spoil the embedding of~$M$. In the proof of \propref{prop:leavepath}, we forced~$Z$ to be covered by the image of the second level path, now we will force~$Z$ to be covered by the images of~$L$.

We again apply \lemref{lem:partition} (Partition Lemma) and partition~$V$ into three parts~$U_F$, $U_M$, and~$U_L$ satisfying~$|U_F|=|V(F)|+4\Delta m$, $|U_M|=|M|+\Delta m$, and $|U_L|=|L|-5\Delta m$. This implies that~$U_F$ is of size at least $20\Delta m$ and that~$U_M$ and~$U_L$ are each of size at least $20\Delta m^2$. Hence, $G[U_F]$ is a $(|U_F|,2\Delta)$-expander. Moreover, for every set~$W_L$ of size~$|L|$ with $U_L\subseteq W_L$ we have for all sets~$X\subseteq V$ with~$1\le |X|< m$, that
\begin{equation}
\label{eq:lmexpansion}
|N_G(X)\cap W_L|\ge |N_G(X)\cap U_L|\ge 2\Delta m|X|
\end{equation}
and for every set~$W_M$ of size~$|M|$ with $W_M\subseteq U_M$ we have for all sets~$X\subseteq V$ with~$1\le |X|< m$, that
\begin{equation}
\label{eq:mexpansion}
|N_G(X)\cap W_M|\ge |N_G(X)\cap U_M|-\Delta m\ge \Delta m|X|.
\end{equation}

The construction of the embedding~$\phi$ closely follows that of the previous proposition. We first apply \corref{cor:almostspanning} (Almost Spanning Tree Embedding) to find an embedding~$\phi_F$ of~$F$ in~$U_F$. Let~$W_F:=\phi_F(F)$. We later move~$U_F\setminus W_F$ to~$U_L$.

Next, we give an embedding~$\phi_M$ of~$M$. Let~$Z$ be the exceptional vertices in~$M$, \ie, let $Z:=\{u\in U_M\mid |N_G(u)\cap\phi_F(K)|<m\}$. Then~$|Z|<m$ by \lemref{lem:exceptionalset} (Small Exceptional Sets). Let~$W_M$ be an arbitrary subset of~$U_M\setminus Z$ of size~$|M|$ and let~$W_L$ be the set given by~$V\setminus(W_F\cup W_M)$. Then, by the choice of~$W_M$ and by~(\ref{eq:lmexpansion}) and~(\ref{eq:mexpansion}), we can apply \lemref{lem:starmatching} (Star Matching) to find an embedding~$\phi_M$ of~$M$ in~$U_M$ which respects the edges between~$K$ and~$M$ in~$T$.

Finally, we give an embedding~$\phi_L$ of~$L$. Because of~(\ref{eq:lmexpansion}) and~(\ref{eq:mexpansion}), the preconditions of \lemref{lem:starmatching} (Star Matching) also hold for~$L$. Note that this time there is no exceptional set, since~(\ref{eq:mexpansion}) guarantees the minimum degree constraint~(iii) in \lemref{lem:starmatching}. Hence, we can embed~$L$ such that the edges between~$M$ and~$L$ in~$T$ are respected. This concludes the proof of the proposition and hence of \thmref{thm:universality}.
\end{proof}

\section{Random Graphs}
\label{sec:random}
The typical random graph is one of the most prominent examples of a strong expander. Random graphs a.a.s.\ have the property that between all pairs of sufficiently large vertex sets the edge density is concentrated around its expectation. The following result links this pseudo-random property to $(n,d)$-expansion and allows us to show that binomial random graphs and random regular graphs with sufficiently large (expected) degree are $(n,d)$-expanders.
\begin{lemma}
\label{lem:random}
Let~$C\in\mathbb{R}^+$, let~$(\mathcal{G}_n)_{n\in\mathbb{N}}$ be a sequence of probability distributions over all graphs on $n$ vertices, and let~$d\colon\mathbb{N}\to\mathbb{R}^+$ satisfy~$d\ge 3$. Suppose there exists an absolute constant~$n_0\in\mathbb{N}$ such that, for all $n\ge n_0$, for~$G$ drawn according to~$\mathcal{G}_n$, and for all disjoint sets $X,Y\subseteq V(G)$ that satisfy~$1\le |X|< m(n,d)$ and $|Y|=n-\lceil(d+1)|X|\rceil+1$ or satisfy~$|X|=|Y|=m(n,d)$, it holds that
\begin{equation}
  \label{eq:random}
  \Pr\big[e_G(X,Y)=0\big]\le C n^{-\frac{6d|X||Y|}{n}}.
\end{equation}
Then a random graph drawn according to~$\mathcal{G}_n$ is a.a.s.\ an $(n,d)$-expander.
\end{lemma}

\begin{proof}
Let~$m:=m(n,d)$, let~(\nea) be the event that there exist two disjoint sets $X,Y\subseteq V(G)$ with $1\le |X|< m$ and $|Y|= n-\lfloor(d+1)|X|\rfloor$ for which $e_G(X,Y)=0$ holds, and let~(\neb) be the event that there exist two disjoint sets $X,Y\subseteq V(G)$ with $|X|=|Y|=m$ for which $e_G(X,Y)=0$ holds. 

Clearly, $G$ is an $(n,d)$-expander if it satisfies neither~(\nea) nor~(\neb). Since $6dm^2/n\ge 3m$, we have by~\eqref{eq:random} and the union bound that
\[
\Pr\big[(\neb)\big]\le C\binom{n}{m}\binom{n-m}{m} n^{-3m}\le C n^{-1}.
\]

For~$1\le k< m$, we have
\[
\frac{6dk(n-\lfloor(d+1)k\rfloor)}{n}\ge 6dk-3\lfloor(d+1)k\rfloor\ge 2dk.
\]
 Thus, 
\[
\Pr\big[(\nea)\big]\le C\sum_{k=1}^m\binom{n}{k}\binom{n}{\lfloor dk\rfloor}n^{-2dk}\le C\cdot n\cdot n^k\cdot n^{dk}\cdot n^{-2dk}\le C n^{-1}.
\]
Hence, by the union bound, $\Pr[(\nea)\,\vee\,(\neb)]=o(1)$ and therefore~$G$ a.a.s.\ satisfies~(\ea) and~(\eb), \ie, $G$ is a.a.s.\ an $(n,d)$-expander.
\end{proof}

If the average degree~$p n$ (actually, $p(n-1)$) of the random graph~$G(n,p)$ is of order~$d\log n$, then $G(n,p)$ has expansion~$d$.
\begin{lemma}
\label{lem:gnp}
Let~$d\colon\mathbb{N}\to\mathbb{R}^+$ satisfy $d\ge 3$. Then $G(n,7d n^{-1}\log n)$ is a.a.s.\ an $(n,d)$-expander.
\end{lemma}
Together with \thmref{thm:universality}, this lemma immediately implies \thmref{thm:gnp}.

\begin{proof}
This lemma is a direct consequence of \lemref{lem:random}, since
\[
\Pr\big[e_G(X,Y)=0\big]=(1-p)^{|X||Y|}\le\euler^{-p|X||Y|}
\]
holds for all disjoint vertex sets~$X$ and~$Y$ in a graph~$G$ drawn from $G(n,p)$.
\end{proof}

A second example of strong expanders are random regular graphs. Similar to the random graph~$G(n,p)$, a random regular $n$-vertex graph is a.a.s.\ an $(n,d)$-expander if its degree is of order $d\log n$. Note that our proof is restricted to random regular graphs with relatively large degree (at least of order $\sqrt{n}\log n$).
\begin{lemma}
\label{lem:regular}
Let the function $d\colon\mathbb{N}\to\mathbb{R}^+$ satisfy $d\ge \sqrt{n}\log n$. Let $r\colon\mathbb{N}\to\mathbb{N}$ satisfy that $r n$ is even and that $7d\log n\le r\ll n$. Then the random $r$-regular graph is a.a.s.\ an $(n,d)$-expander.
\end{lemma}
Together with \thmref{thm:universality}, \lemref{lem:regular} implies \thmref{thm:regular}. To prove \lemref{lem:regular}, we follow the proof of Theorem~2.2 in~\cite{KrSuVuWo01} which uses a switching argument originally introduced in~\cite{McWo90}.  
\begin{proof}
Let~$n$ be a sufficiently large integer. Let~$V$ be a vertex set of size~$n$ and let~$X$ and~$Y$ be two sets such that~$|X|=|Y|=m(n,d)$ or~$1\le |X|< m(n,d)$ and $|Y|=n-(d+1)|X|$. By \lemref{lem:random}, it is sufficient to show that
\[
\Pr[e_G(X,Y)=0]\le n^{-\frac{6d|X||Y|}{n}}
\]
in order to prove the theorem.

Let $k:=\lfloor\frac{2r|X||Y|}{n}\rfloor$ and, for~$j\in\{0,\hdots,k\}$, let~$\mathcal{C}_j$ be the class of $n$-vertex $r$-regular graphs on~$V$ that have exactly~$j$ edges with end-vertices in~$X$ and~$Y$. 

First, let $j\in\{1,\hdots,k\}$ and let~$G\in\mathcal{C}_j$. We consider the following procedure. We choose an ordered vertex pair~$(x,y)\in X\times Y$ and two other ordered vertex pairs $(v,w),(v',w')\in V\times V$ such that~$\{x,y\}$, $\{v,w\}$, and~$\{v',w'\}$ are edges of~$G$. We delete these three edges from~$E(G)$ and then add each of the three pairs~$\{x,v\}$, $\{y,v'\}$, and~$\{w,w'\}$ to~$E(G)$, provided each of them forms a new edge in~$G$. If the resulting graph is in~$\mathcal{C}_{j-1}$, we call this procedure a \emph{forward switching}. The total number of forward switchings in~$G$ is at most
\[
j r^2n^2
\]
since there are exactly~$j$ choices for~$(x,y)$ and at most~$2|E(G)|=rn$ choices each for $(v,w)$ and $(v',w')$.

Next, also for $j\in\{1,\hdots,k\}$, let~$G\in\mathcal{C}_{j-1}$ and consider the following procedure. We choose three ordered vertex pairs~$(x,v)\in X\times V$, $(y,v')\in Y\times V$, and~$(w,w')\in V\times V$, such that~$\{x,v\}$, $\{y,v'\}$, and~$\{w,w'\}$ are edges of~$G$. We delete all three edges from~$E(G)$ and add each of the three pairs $\{x,y\}$, $\{v,w\}$, and~$\{v',w'\}$ to~$E(G)$, provided each of them forms a new edge in~$G$. If the new graph is in~$\mathcal{C}_j$, we call this procedure a \emph{reverse switching}.

There are $|X||Y|-j=(1-o(1))|X||Y|$ ways to choose vertices~$(x,y)\in X\times Y$ such that~$\{x,y\}\notin E(G)$. Then there are at least~$rn-j-2r=(1-o(1))rn$ ways to choose the pair~$(w,w')$, since we can choose (and count twice) any edge which does not run between~$X$ and~$Y$ and which is not incident to~$x$ or~$y$. Finally, there are at least $(r-3)^2-2j=(1-o(1))r^2$ ways to choose the vertices~$v$ and~$v'$, since we can choose any combination of a neighbor of~$x$ and a neighbor of~$y$ except for~$w$, $w'$, and $v$ or~$v'$, respectively, and except for the~$j$ edges between~$X$ and~$Y$. Therefore, for sufficiently large~$n$, the total number of reverse switchings in~$G$ is at least
\[
(1-o(1))|X||Y|r^3n
\]

Every pair of a forward switching and a graph in~$\mathcal{C}_j$ can be identified with the corresponding pair of a backward switching and a graph in~$\mathcal{C}_{j-1}$. Thus,
\[
\frac{j r^2n^2}{4}|\mathcal{C}_j|\ge(1-o(1))\frac{|X||Y|r^3n}{2}|\mathcal{C}_{j-1}|.
\]
Hence,
\[
|\mathcal{C}_j|\ge(1-o(1))\frac{k}{j} |\mathcal{C}_{j-1}|
\]
and therefore
\[
|\mathcal{C}_k|\ge(1-o(1))^k\frac{k^k}{k!}|\mathcal{C}_0|\ge\euler^{(1-o(1))k}|\mathcal{C}_0|
\]

Thus, for sufficiently large~$n$,
\[
\Pr[e_G(X,Y)=0]=\Pr[G\in\mathcal{C}_0]\le\euler^{-(1-o(1))k}\le n^{-\frac{6d|X||Y|}{n}}
\]
which concludes the proof of \thmref{lem:regular}.
\end{proof}

\section{Locally Sparse Expanders}
\label{sec:sparse}

So far, the only examples of $(n,d)$-expanders that we considered were the Binomial random graph and random regular graphs. Note that even if $d \in O(1)$, then $G(n,p)$ becomes an $(n,d)$-expander only when $p \gg \log n/n$. But when the edge probability is so high, then a.a.s.~$G(n,p)$ also contains a triangle. Nevertheless, in this subsection we will see that there exist $(n,d)$-expanders with almost logarithmic girth. Moreover, we will show that even for~$d$ slightly smaller than~$\sqrt{n}$, there still exists $(n,d)$-expanders which are triangle-free, while a random graph with the same expansion already contains a copy of~$K_5$.

\begin{definition}[$(k,\ell)$-locally sparse]
We call a graph \emph{$(k,\ell)$-locally sparse} if all of its induced $k$-vertex subgraphs have at most~$\ell$ edges.
\end{definition}

For example, a $(r+1,\binom{r+1}{2})$-locally sparse graph is~$K_{r+1}$-free, \ie, has clique number at most~$r$, and a connected $(r+1,r+1)$-locally sparse graph does not contain a cycle of length at most~$r+1$, \ie, has girth at least~$r+2$. Note that, for~$r=1$, the two previous observations coincide. 
\begin{lemma}
\label{lem:sparse}
There exists an absolute constant~$c\in\mathbb{R}^+$ such that the following statement holds. Let~$n,k,\ell\in\mathbb{N}$ with~$\ell\ge 2$ and let~$d:=c n^{1-(k-2)/(\ell-1)}\log^{-1}n$ satisfy~$d\ge 3$. Then there exists an $(n,d)$-expander~$H$ that is $(k,\ell)$-sparse. Moreover, $H$ satisfies
\begin{equation}
\label{eq:sparse}
e_H(X,Y)\ge\frac{48d|X||Y|\log n}{n}
\end{equation}
for all (not necessarily disjoint) sets $X,Y\subseteq V(H)$ which satisfy~$1\le |X|< m(n,d)$ and $|Y|=n-\lceil(d+1)|X|\rceil+1$ or satisfy~$|X|=|Y|=m(n,d)$.
\end{lemma}

If we apply \lemref{lem:sparse} with~$k=r+1$ and~$\ell=\binom{r+1}{2}$, then the resulting graph has clique number at most~$r$. This, together with \thmref{thm:universality}, implies \thmref{thm:sparse}. Note that if we were able to improve the bound on~$d$ in \thmref{thm:universality} to $c n^{1/r}\log^{-1} n$, then \lemref{lem:sparse} with~$k=r+1$ and~$\ell=r+1$ would immediately give the existence of tree-universal graphs with girth at least~$r+2$.

Before proving \lemref{lem:sparse}, we first state the well-known Chernoff's inequality (see, e.g.,~\cite[Theorem~2.3]{JaLuRu00}) which we use to bound the tail probabilities of the binomial distribution~$\Bin(n,p)$. Recall that, for a random variable~$X\sim\Bin(n,p)$, we have $\Pr[X=k]=\binom{n}{k}p^k(1-p)^{n-k}$ for all~$0\le k\le n$ and~$\EXP[X]=p n$.
\begin{theorem}[Chernoff's inequality]
\label{thm:bin}
Let~$\eps$ be a positive constant satisfying~$\eps\le 3/2$ and let~$X\sim\Bin(n,p)$. Then
\[
\Pr\big[|X-\EXP[X]|>\eps\big]\le\euler^{-\frac{\eps^2}{3} \EXP[X]}.
\]
\end{theorem}
We are now ready to prove \lemref{lem:sparse}. The main idea of this proof originates in~\cite{Kr95}.
\begin{proof}[Proof of \lemref{lem:sparse}]
In order to prove the lemma, we will show that the random graph~$G(n,p)$ with $p:=144d n^{-1}\log n$ is a.a.s.\ an $(n,d)$-expander, even if we remove all edges of some maximal family of edge-disjoint subgraphs on $k$ vertices and~$\ell$ edges.

Let~$c\in\mathbb{R}^+$ be small enough such that~$p\le\euler{}^{-7}n^{-(k-2)/(\ell-1)}$ and assume that~$n$ is sufficiently large. Furthermore, let~$\mathcal{F}$ be the class of all graphs with~$k$ vertices and~$\ell$ edges and, for every $n$-vertex graph~$G$, let~$\mathcal{F}(G)$ be an arbitrary (but fixed) family of edge-disjoint copies of graphs in~$\mathcal{F}$ in~$G$.

Now, let~$G$ be a random graph drawn according to~$G(n,p)$ and let~$H$ be the random graph resulting from~$G$ by removing all edges covered by~$\mathcal{F}(G)$. Because the family~$\mathcal{F}(G)$ is maximal, $H$ is $(k,\ell)$-sparse.

For all (not necessarily disjoint) vertex sets~$X$ and~$Y$ in~$H$ which satisfy~$1\le |X|< m(n,d)$ and $|Y|=n-\lceil(d+1)|X|\rceil+1$ or which satisfy~$|X|=|Y|=m(n,d)$, we are going to show that a.a.s.~(\ref{eq:sparse}) in \lemref{lem:sparse} holds. Clearly, this implies that~$H$ is an $(n,d)$-expander and that there indeed exists a graph which satisfies the properties stated in \lemref{lem:sparse}.

Let $X,Y\subseteq V(H)$ be two sets which satisfy~$1\le |X|< m(n,d)$ and $|Y|=n-\lceil(d+1)|X|\rceil+1$ or which satisfy~$|X|=|Y|=m(n,d)$. Then, by Chernoff's inequality (\thmref{thm:bin}), we have
\[
\Pr\Big[e_G(X,Y)\le\frac{p|X|(|Y|-1)}{2}\Big]\le \euler^{-\frac{p|X|(|Y|-1)}{12}}.
\]
Note that the~$-1$ in the product~$|X|(|Y|-1)$ is due to the fact that~$X$ and~$Y$ may intersect and~$G$ is simple, \ie, has no loops. Since~$|Y|\ge 5$ and thus $5(|Y|-1)\ge 4|Y|$ for sufficiently large~$n$, this implies
\begin{equation}
\label{eq:chernoffpart}
\Pr\Big[e_G(X,Y)\le\frac{2p|X||Y|)}{5}\Big]\le \euler^{-\frac{p|X||Y|}{15}}.
\end{equation}

Next, let~$\mathcal{F}(X,Y)$ be the family of sets in $\mathcal{F}(G)$ which contain at least one edge between~$X$ and~$Y$. For every~$t\in\mathbb{N}$, there are at most~$\binom{|X||Y|}{t}n^{(k-2)t}$ candidates for a collection of~$t$ edge-disjoint copies in~$G$ of graphs in~$\mathcal{F}(G)$ which cover at least one edge between~$X$ and~$Y$: $\binom{|X||Y|}{t}$ choices for the $t$ different (potential) edges between~$X$ and~$Y$ and at most~$n^{(k-2)t}$ choices for the remaining~$k-2$ vertices of each of the copies. Since the~$t$ copies are edge-disjoint and all graphs in~$\mathcal{F}(G)$ have~$\ell$ edges, each candidate collection occurs with probability
\[
p^{\ell t}=p^t p^{(\ell-1)t}\le\Big(\frac{p}{\euler^{7(\ell-1)}}\Big)^t n^{-(k-2)t}
\]
since $p\le\euler{}^{-7}n^{-(k-2)/(\ell-1)}$. Now, for~$\ell\ge 2$ we have
\[
\euler^{7(\ell-1)}=\euler^{6\ell-7}\euler{}^{\ell}\ge 15\euler\ell\euler{}^{\ell}
\]
and therefore
\[
p^{\ell t}\le\Big(\frac{p}{15\euler\ell}\Big)^t\euler^{-\ell t}n^{-(k-2)t}.
\]
Thus, by the union bound,
\[
\Pr\Big[|\mathcal{F}(X,Y)|\ge t\Big] \le\binom{|X||Y|}{t}n^{(k-2)t}\Big(\frac{p}{15\euler\ell}\Big)^t\euler^{-\ell t}n^{-(k-2)t}\le\Big(\frac{p|X||Y|}{15\ell t}\Big)^t\euler^{-\ell t}
\]
and, for~$t=\frac{p|X||Y|}{15\ell}$, this implies that
\begin{equation}
\label{eq:deletionpart}
\Pr\Big[ |\mathcal{F}(X,Y)|\ge\frac{p|X||Y|}{15\ell}\Big]\le \euler^{-\frac{p|X||Y|}{15}}
\end{equation}
Observe that by deleting the edges contained in a graph from $\mathcal{F}(X,Y)$ we can reduce $e_G(X,Y)$ by at most $\ell$. Thus, by again applying the union bound to~(\ref{eq:chernoffpart}) and~(\ref{eq:deletionpart}), we get
\[
\Pr\Big[e_H(X,Y)\le\frac{p|X||Y|}{3}\Big]\le\Pr\Big[e_G(X,Y)\le\frac{2p|X||Y|}{5}\;\vee\;|\mathcal{F}(X,Y)|\ge\frac{p|X||Y|}{15\ell}\Big]\le 2\euler^{-\frac{p|X||Y|}{15}}.
\]
By substituting~$p=144d n^{-1}\log n$, we obtain
\[
\Pr\Big[e_H(X,Y)<\frac{48d|X||Y|\log n}{n}\Big]\le 2n^{-\frac{48d|X||Y|}{5n}}.
\]
Now, since $48/5\ge 6$, \lemref{lem:sparse} follows from the same union-bound argument over all sets~$X$ and~$Y$ as in \lemref{lem:random}.
\end{proof}

Using \lemref{lem:sparse}, we can now prove \lemref{lem:radius} which was presented in the introduction.
\begin{proof}[Proof of \lemref{lem:radius}]
Let~$c$ be any positive constant that is smaller than that in \lemref{lem:sparse} and which satisfies~$c\le 1/60$. We assume that $n$ is sufficiently large. Let~$d:=c n^{1/r}\log^{-1}n$ and note that~$d\ge 3$ by the choice of~$r$. 

For~$r=1$, let~$H$ be the complete graph~$K_n$ and for~$r\ge 2$ let~$H$ be the graph from \lemref{lem:sparse} with~$k=\ell=r+1$. In both cases~$H$ has girth~$r+2$ (note that the graphs constructed in \lemref{lem:sparse} are expanders and thus they are always connected). Let~$G$ be a random graph drawn uniformly at random from~$\mathcal{G}_H$. Thus, independently for every~$\{v,w\}\in E(H)$, the graph~$G$ contains either the two edges~$\{u_v,u_w\}$ and~$\{u'_v,u'_w\}$ or the two edges~$\{u_v,u'_w\}$ and~$\{u_v,u'_w\}$, where each choice has probability~$1/2$.

Our first observation is that the radius of~$G$ is at least as large as the girth of~$H$, \ie, at least~$r+2$, since each path between two vertices~$u_v$ and~$u'_v$ in~$G$ corresponds to a non-trivial closed walk in~$H$. Next, we show that~$G$ is a.a.s.\ a $(2n,d)$-expander.

Next, let~$X$ and~$Y$ be two disjoint sets in~$V(G)$ which satisfy~$1\le |X|< m(2n,d)$ and $|Y|=n-\lceil(d+1)|X|\rceil+1$ or which satisfy~$|X|=|Y|=m(2n,d)$. We need to show that
\[
\Pr[e_G(X,Y)=0]\le n^{-\frac{6d|X||Y|}{n}}.
\]
Let $A:=\{v\in V(H) \mid u_v\in X\,\vee\,u'_v\in X\}$ and $B:=\{v\in V(H) \mid u_v\in Y\,\vee\,u'_v\in Y\}$ be the projections of~$X$ and~$Y$ to~$V(H)$. Note that~$A$ and~$B$ may intersect. Clearly, we have $|A|\ge|X|/2$ and~$|B|\ge|Y|/2$. Consider all ordered pairs~$(v,w)\in A \times B$. For each of them mark an arbitrary pair from $\{u_v,u'_v\}\times\{u_w,u'_w\}$ in~$X\times Y$. Then there exist exactly~$e_H(A,B)$ marked pairs, each of which forms an edge in~$G$ between~$X$ and~$Y$ independently with probability~$1/2$. Thus,
\begin{equation}
\label{eq:radius}
\Pr[e_G(X,Y)=0]\le 2^{-e_H(A,B)}\le e^{-\frac{e_H(A,B)}{2}}.
\end{equation}

For the case~$r=1$, we have $H=K_n$. Then, since~$e_H(A,B)$ counts edges between ordered pairs and since~$A$ and~$B$ may intersect, we have for sufficiently large~$Y$ that
\[
e_H(A,B)=|A||B|-|A|=\frac{|X||Y|-2|X|}{4}\ge\frac{|X||Y|}{5}
\]
and therefore, since~$d\le\frac{1}{60}n\log^{-1}n$, we have
\[
\Pr[e_G(X,Y)=0]\le e^{-\frac{|X||Y|}{10}}\le n^{-\frac{6d|X||Y|}{n}}.
\]

For the case~$r\ge 2$, $H$ is the graph provided by \lemref{lem:sparse}. Since~$H$ satisfies~(\ref{eq:sparse}), \ie,
\[
e_H(A,B)\ge\frac{48d|A||B|\log n}{n}
\]
we get, by inequality~(\ref{eq:radius}) and by ~$|A|\ge|X|/2$ and~$|B|\ge|Y|/2$, that
\[
\Pr\Big[e_G(X,Y)=0\Big]\le n^{-\frac{6d|X||Y|}{n}}.
\]

 In both cases, we have 
\[
\Pr\Big[e_G(X,Y)=0\Big]\le n^{-\frac{6d|X||Y|}{n}}\le(2n)^{-\frac{6d|X||Y|}{2n}}
\]
and the \lemref{lem:radius} follows from \lemref{lem:random}.
\end{proof}

\section{The Maker-Breaker Expander Game}
\label{sec:makerbreaker}
This section is devoted to the proof of \thmref{thm:mbuniversal}, which we presented in the introduction. To this end, we first formulate the \emph{Maker-Breaker Expander Game}, in which Maker tries to claim a subset of the edges that induces an $(n,d)$-expander.
\begin{definition}[Maker-Breaker Expander Game]
Let~$n\in\mathbb{N}$, let $d\in\mathbb{R}^+$, and let~$G$ be a graph. Then the \emph{Maker-Breaker $(n,d)$-expander game on~$G$} is the Maker-Breaker game on the hypergraph $(E(G),\mathcal{F})$, where $\mathcal{F}$ consists of all edge sets~$F\subseteq E(G)$ such that the subgraph $(V(G),F)$ is an $(n,d)$-expander.
\end{definition}
If the $(1\,{:}\,b)$ Maker-Breaker expander game is played on an $(n,15bd\log n)$-expander, then Maker can always secure the edges of an $(n,d)$-expander.
\begin{theorem}
\label{thm:makerbreaker}
There exists an absolute constant~$n_0\in\mathbb{N}$ such that the following statement holds. Let~$n,b\in\mathbb{N}$ and~$d\in\mathbb{R}^+$ satisfy~$n\ge n_0$ and~$d\ge 3$. Then the $(1\,{:}\,b)$ Maker-Breaker $(n,d)$-expander game is Maker's win on every $(n,15 b d\log n)$-expander.
\end{theorem}
\thmref{thm:makerbreaker}, together with \thmref{thm:universality}, implies \thmref{thm:mbuniversal}. The key to the proof of this theorem is Beck's generalization of the Erd\H{o}s-Selfridge criterion for Breaker's win~\cite{ErSe73,Be82}.
\begin{theorem}
\label{thm:beck}
Let~$a,b\in\mathbb{N}$. Let~$(X,\mathcal{F})$ be a finite hypergraph. Suppose that
\[
\sum_{F\in \mathcal{F}}(1+b)^{-|F|/a}<\frac{1}{1+b}.
\]
Then the $(a:b)$ Maker-Breaker game on~$(X,\mathcal{F})$ is Breaker's win.
\end{theorem}

\begin{proof}[Proof of \thmref{thm:makerbreaker}]
  Let~$n,d_0,d,b\in\mathbb{N}$ satisfy~$d\ge 3$ and $d_0\ge 15 b d\log n$. Let~$m:=m(n,d)$, let~$m_0:=m(n,d_0)$, and let~$G$ be an $(n,d_0)$-expander. Then we have $d_0\ge 10bd\log_2 n$ and also $m_0\le\frac{n}{20 b d \log_2 n}$ for sufficiently large~$n$. 

Since \thmref{thm:beck} only allows us to bound the bias for Breaker's win, we reformulate the $(1\,{:}\,b)$ Maker-Breaker $(n,d)$-expander game on~$G$ so that the roles of Maker and Breaker are reversed. For this, consider the hypergraph~$(E(G),\mathcal{F})$ where~$F\in\mathcal{F}$ if~$F$ consists of all edges of~$G$ between two disjoint sets~$X\subseteq V(G)$ and~$Y\subseteq V(G)$ such that either
\begin{enumerate}[(i)]
\item $1\le |X|< m$ and~$|Y|=n-\lceil(d+1)|X|\rceil+1$, or
\item $|X|=|Y|=m$.
\end{enumerate}

Consider the outcome of one game of the $(b\,{:}\,1)$ Maker-Breaker game on~$(E(G),\mathcal{F})$ and let~$H$ be the subgraph of~$G$ defined by the edges claimed by Breaker.

If Breaker wins, then~$H$ satisfies the conditions~(\ea) and~(\eb) in \defref{def:expander} and~$H$ is an $(n,d)$-expander. Thus, if we show that the generalized Erd\H{o}s-Selfridge criterion,
\[
\sum_{F\in \mathcal{F}}2^{-|F|/b}<\frac{1}{2},
\]
holds for~$\mathcal{F}$, then the $(b\,{:}\,1)$ Maker-Breaker game on~$(E(G),\mathcal{F})$ is \emph{Breaker's win} and therefore the $(1\,{:}\,b)$ Maker-Breaker $(n,d)$-expander game on~$G$ is \emph{Maker's win}. Also note that because of the exchanged roles Breaker starts the game. This, however, only strengthens the result.

For~$k,\ell\in\{1,\dots,n\}$, let~$\mathcal{F}(k,\ell)$ be the family of edge sets~$F\in\mathcal{F}$ such that~$F$ contains all edges in~$G$ between a set~$X$ of size~$k$ and a set~$Y$ of size~$\ell$. Then, for every~$F\in\mathcal{F}$, either $F\in\mathcal{F}\big(k,n-(d+1)k+1\big)$ holds for some~$k\in \{1,\dots,m\}$ or it holds that $F\in\mathcal{F}(m,m)$. Therefore,
\[
\sum_{F\in \mathcal{F}}2^{-|F|/b}=\sum_{k=1}^{m}\sum_{F\in\mathcal{F}(k,n-(d+1)k+1)}2^{-|F|/b}+\sum_{F\in\mathcal{F}(m,m)}2^{-|F|/b}.
\]

Let~$X$ and~$Y$ be two disjoint subsets of~$V(G)$ with $1\le |X|< m$ and $|Y|=n-(d+1)|X|+1$. First, suppose that~$|X|< m_0$. Since~$G$ is an $(n,d_0)$-expander, we have~$|N_G(X)|\ge d_0|X|$. Thus, since $d\ge 3$,
\[
e_G(X,Y)\ge |N_G(X)\cap Y|\ge d_0|X|-d|X|\ge b(d+2)|X|\log_2 n.
\]
Next, suppose that $m_0\le |X|< m$. Since~$d\ge 3$ and thus~$|Y|\ge \frac{n}{3}$, we have by \lemref{lem:density} that
\[
e_G(X,Y)\ge\frac{|X||Y|}{4m_0}\ge\frac{n}{12m_0}|X|\ge b(d+2)|X|\log_2 n.
\]
Thus, we have. for~$d\ge 3$,
\[
\sum_{k=1}^{m}\sum_{F\in\mathcal{F}(k,n-(d+1)k+1)}2^{-|F|/b}\le\sum_{k=1}^{m}\binom{n}{k}\binom{n}{dk}n^{-(d+2)k}\le 2n^{-1}.
\]

Finally, suppose that $|X|=|Y|=m$. Then, again by \lemref{lem:density}, we have
\[
e_G(X,Y)\ge\frac{|X||Y|}{4m_0}\ge\frac{m n}{8d m_0}\ge \frac{5}{2}b m\log_2 n.
\]
Therefore,
\[
\sum_{F\in\mathcal{F}(m,m)}2^{-|F|/b}\le\binom{n}{m}\binom{n}{m}n^{-5m/2}\le n^{-1/2}.
\]

Hence,
\[
\sum_{F\in\mathcal{F}}2^{-|F|/b}\le\frac{3}{n^{1/2}}<\frac{1}{2},
\]
and Beck's generalization of the Erd\H{o}s-Selfridge criterion is indeed satisfied.
\end{proof}

\section{Conclusion}
We have shown that, for sufficiently large~$n$, every $(n,8n^{2/3}\max\{\Delta,\log n\})$-expander is universal for the class of all $n$-vertex trees with maximum degree at most~$\Delta$. This implies that binomial random graphs and random regular graphs with sufficiently large (average) degree are a.a.s.\ $\mathcal{T}(n,\Delta)$-universal. Our result also leads to constructions of locally sparse $\mathcal{T}(n,\Delta)$-universal graphs. We have also discussed $\mathcal{T}(n,\Delta)$-universality in the setting of the Maker-Breaker game.

One major open problem is to establish the smallest value of~$p$ for which $G(n,p)$ becomes a.a.s.\ $\mathcal{T}(n,\Delta)$-universal. Here, our work leaves a substantial gap of~$n^{2/3}$ compared to the lower bound in~\cite{Kr10}. Also, it would be interesting to see why the corresponding lower bound for $(n,d)$-expanders in \thmref{thm:lower} differs so drastically from that in~\cite{Kr10} and to possibly find pseudo-random sufficient conditions which do not yield this discrepancy. In the spirit of \thmref{thm:sparse}, it would be nice to see constructions of tree-universal graphs which are triangle-free or even have large girth. Finally, although our embedding results are (for the most part) constructive, they do not give an efficient algorithm to find the embeddings. Here, an algorithmic version would be also desirable.

\bibliographystyle{amsplain}
\bibliography{abbreviations,proceedings,tree-universality}

\end{document}